\documentclass[12pt,reqno]{amsart}
\usepackage[a4paper,margin=2.2cm]{geometry}

\usepackage{microtype}
\usepackage{tikz,tikz-cd}
\usepackage[cal=euler,scr=dutchcal]{mathalfa}

\input{commands_new.tex}
\defcite\MR{meinhardt_donaldson}
\defcite\GS{gottsche_perverse}
\defcite\BGS{beilinson_koszul}
\defcite\BBD{BBD}
\defcite\sai{saito_introduction}
\defcite\lep{lepotier_lectures}

\operl{GL,SL,Ext,GIT,Log,Perv,Sm,Smp,Var,Gr,deg,codim,supp,Hom,Vect,Coh,Rep,Fun,ch,Id,Exp,im,IH,MHM,MHS,rat,CM,DA}
\operl{BC,JH,HV,BPS,MMHM,DTV,HH,BL,DM,LT,amp,HS,HM}
\def\Vecf{\opn{Vec}}

\def\vir{{\rm vir}}
\def\nil{{\rm nil}}

\def\mot{\mathfrak{m}}

\oper{IC}
\opr\ICV{\mathbf{IC}}

\def\aDT{\opn{DT}} 
\def\aDTm{\aDT^\mot} 
\def\DT{\opn{\mathbf{DT}}} 
\def\DTm{\DT^\mot} 
\def\HDT{\opn{DT}^{E}}

\def\K{\mathbf K} 
\def\pHS{\HS^p} 
\def\pHM{\HM^p} 

\def\f{\mathsf f} 
\def\st{\mathsf s} 
\def\M{\mathcal M} 
\def\Ms{\M^\st} 
\def\Mf{\M^\f} 
\def\MM{\mathfrak M}

\def\p{{}^{\cscale[.8]{\mathrm p}}\!} 
\def\Gv{\bL^\oh-\bL^{-\oh}}

\def\Lap{\La_+}
\def\bLa{\boldsymbol\Lambda}
\def\bga{\boldsymbol\gamma}
\def\m{\mathbf{m}} 

\def\sr{\mathsf r} 
\def\sd{\mathsf d} 
\def\bk{K} 
\def\ohh{{1/2}}

\usepackage[colorlinks,allcolors=blue]{hyperref}

\def\ohh{{1/2}}

\begin{document}
\title[Intersection cohomology of moduli spaces]{Intersection cohomology of moduli spaces of vector bundles over curves}
\author{Sergey Mozgovoy}
\author{Markus Reineke}

\begin{abstract}
We compute the intersection cohomology of the moduli spaces $\M_{r,d}$ of semistable vector bundles
having rank $r$ and degree $d$ over a curve.
We do this by relating the Hodge-Deligne polynomial of the intersection cohomology of $\M_{r,d}$ to the Donaldson-Thomas invariants of the curve.
These invariants can be computed by methods going back to Harder, Narasimhan, Desale and Ramanan.
More generally, we introduce Donaldson-Thomas classes in the Grothendieck group of mixed Hodge modules over $\M_{r,d}$ and relate them to the class of the intersection complex of $\M_{r,d}$.
Our methods can be applied to the moduli spaces of semistable objects in arbitrary hereditary categories.
\end{abstract}
\maketitle
\section{Introduction}
Let $C$ be a complex smooth projective curve of genus $g$.
For $\ga=(r,d)\in\bZ_{>0}\xx\bZ$, let $\M_\ga$
(resp.\ $\Ms_\ga$)
denote the moduli space of semistable (resp.\ stable) vector bundles of rank $r$ and degree $d$ over $C$.
Then $\M_\ga$ is projective and irreducible,
$\Ms_\ga\sbe\M_\ga$ is open and smooth,
and $\Ms_\ga=\M_\ga$ if $r$~and~$d$ are coprime.
If $\Ms_\ga$ is nonempty (this is always the case for $g\ge2$),
then its dimension is $(g-1)r^2+1$.
The history of computation of various invariants of $\M_\ga$ for coprime $r$ and $d$ is rather long.
A recursive formula to determine the Betti numbers of $\M_\ga$
was proved in \cite{harder_cohomology,desale_poincare}
by applying the Weil conjectures
and the Siegel formula \cite{harder_cohomology}
based on the computation of the Tamagawa number of $\SL_n$ over function fields \cite{weil_adeles}.
An alternative proof of this recursive formula was obtained in \cite{atiyah_yang} by using gauge theory.
An algebro-geometric method to prove the Siegel formula was developed
in \cite{bifet_abel-jacobi,ghione_effective}
and this method was used in \cite{bano_chow} to prove the recursive formula for the motivic classes and Hodge polynomials of $\M_\ga$.
Independently, the same recursive formula for the Hodge polynomials of $\M_\ga$ was proved in \cite{earl_hodge} by using equivariant cohomology and a refinement of the Atiyah and Bott approach \cite{atiyah_yang}.
On the other hand, the above recursive formula was solved in \cite{laumon_langlands,zagier_elementary},
hence one obtained a rather explicit formula for the invariants of $\M_\ga$ in the case of coprime $r$ and $d$.

\medskip
In this paper, we compute intersection cohomology of $\M_\ga$ for arbitrary $\ga=(r,d)$.
These invariants were previously determined for $r=2$ in
\cite{kirwan_homology,kirwan_corrigendum},
using partial desingularizations of GIT quotients and related techniques developed in
\cite{kirwan_cohomology,kirwan_partial,kirwan_rational}.
Based on these computations,
as well as computations of intersection cohomology for moduli spaces of low-rank vector bundles on surfaces
\cite[Remark 4.6]{yoshioka_bettia},
it was proposed in \cite{manschot_bps} that intersection cohomology should be related to Donaldson-Thomas (DT) invariants,
also known as BPS invariants.
These invariants are usually defined for moduli spaces of sheaves on $3$-Calabi-Yau varieties or, more generally, for moduli spaces of objects of $3$-Calabi-Yau categories \cite{kontsevich_stability}.
But the construction can be also applied to hereditary categories (abelian categories with vanishing $\Ext^i$ for $i\ge2$)
and, in particular, to the category of coherent sheaves over a curve.
The corresponding DT-invariants $\HDT_{\ga}$ can be computed for arbitrary $\ga$ using the recursive formula mentioned earlier.
After relating these invariants to intersection cohomology of $\M_\ga$, we obtain an effective method to compute the latter.

\medskip

More precisely, the moduli space $\M_\ga$
can be represented as a GIT quotient of a smooth variety~$R_\ga=R_{r,d}$ by
an action of a general linear group $G_{\ga}=G_{r,d}$  \cite{newstead_introduction}.
The cohomology with compact support $H^*_c(R_{\ga},\bQ)$ and $H^*_c(G_{\ga},\bQ)$ can be equipped with mixed Hodge structure
and we can consider
the corresponding Hodge-Deligne polynomials (or $E$-polynomials, see \S\ref{sec:HD})
$$E(R_{\ga}),\ E(G_{\ga})\in \bZ[u^{\pm1},v^{\pm1}].$$
For any $\mu\in\bQ$, we consider the generating series
\begin{equation}\label{ser Q1}
Q_\mu=1+\sum_{d/r=\mu}\bL^{(1-g)r^2/2}\frac{E(R_{r,d})}{E(G_{r,d})}t^r
\in\bQ(u^\oh,v^\oh)\pser t,
\end{equation}
where $\bL=E(\bA^1)=uv$ and $\bL^{1/2}=-(uv)^{1/2}$.
An explicit formula for $Q_\mu$ can be obtained using the methods
from \cite{laumon_langlands,zagier_elementary} mentioned earlier
(see Theorem \ref{zagier}).
In particular, if $r$ and~$d$ are coprime, then
$\frac{E(\M_\ga)}{\bL-1}=\frac{E(R_{\ga})}{E(G_{\ga})},$
hence we can determine Betti and Hodge numbers of~$\M_\ga$.
The role of the other summands of the series $Q_\mu$
is less transparent.
Motivated by the definition of DT invariants in other contexts
(see \eg \cite{mozgovoy_motivic,mozgovoy_motivica}),
we define the DT invariants $\HDT_\ga=\HDT_{r,d}\in\bQ(u^\oh,v^\oh)$
by the formula
\begin{equation}\label{int:DT}
\sum_{d/r=\mu}\HDT_{r,d}t^r
=(\Gv)\Log(Q_\mu),
\end{equation}
where $\Log$ is the plethystic logarithm \eqref{Exp-Log}.
As $Q_\mu$ can be computed explicitly,
we can also compute the DT invariants.
These computations show that $\HDT_{r,d}$ are polynomials with integer coefficients.
A similar statement in the context of quiver representations was proved in \cite{kontsevich_stability,efimov_cohomological}.
\medskip

Now let us describe the relationship between DT invariants and intersection cohomology of moduli spaces.
Given a complex algebraic variety $X$ of dimension $d$,
let $\IC_X\in\Perv(\bQ_X)$ be its intersection complex.
It can be lifted to a mixed Hodge module $\ICV_X$ of weight $0$ (see \S\ref{sec:MHM}) so that $\ICV_X=\bQ_X\ang{d}=\bQ_X(d/2)[d]$ for smooth $X$.
The intersection cohomology
\begin{equation}\label{i-coh1}
\IH^*(X,\bQ)=H^*(X,\ICV_X\ang{-d})
\end{equation}
is a complex of mixed Hodge structures (pure of weight zero if $X$ is projective) and we may consider its Hodge-Deligne polynomial.

\begin{theorem}[See Theorem \ref{main1-proof}]
\label{main1}
If $\Ms_\ga=\es$, then $\HDT_\ga=0$.
If $\Ms_\ga\ne\es$, then
\begin{gather}
\HDT_\ga=E(H^*(\M_\ga,\ICV_{\M_\ga}))
=\bL^{-\dim \M_\ga/2}E(\IH^*(\M_\ga,\bQ))\\
\HDT_{\ga}(y,y)
=(-y)^{-\dim \M_\ga}\sum\nolimits_k\dim\IH^k(\M_\ga,\bQ)(-y)^k.
\end{gather}
\end{theorem}

The object $\ICV_X$ is self-dual with respect to Verdier duality,
hence we obtain the following result.

\begin{corollary}
We have
\begin{enumerate}
\item $\HDT_{\ga}\in\bZ[u,v,(uv)^{-\oh}]$.
\item $\HDT_{\ga}(u\inv,v\inv)=\HDT_{\ga}(u,v)$.
\item $\HDT_{\ga}(-y,-y)\in\bN[y^{\pm1}]$.
\end{enumerate}
\end{corollary}

In order to prove Theorem \ref{main1}, we introduce
relative DT classes
$\DT_{\ga}\in\K(\MHM(\M_\ga))$
such that $E(a_!\DT_\ga)=\HDT_\ga$ for the projection $a:\M_\ga\to\pt$
(we also have $E(a_!\ICV_{\M_\ga})=E(H^*(\M_\ga,\ICV_{\M_\ga}))$).
We prove an analogue of Theorem \ref{main1} for these classes.

\begin{theorem}[See Corollary \ref{cor-main2}]
\label{main2}
We have
$\DT_\ga=\begin{cases}
0&\Ms_\ga=\es\\
[\ICV_{\M_\ga}]&\Ms_\ga\ne\es
\end{cases}$
in $\K(\MHM(\M_\ga))$.
\end{theorem}

For the proof of this result we will utilize ideas of
\cite{meinhardt_donaldson}, where moduli spaces of quiver representations were studied;
similar extensions of the methods of \cite{meinhardt_donaldson} were developed independently in \cite{meinhardt_donaldsona}.
In addition, we clarify technical issues in
\cite{meinhardt_donaldsona,meinhardt_donaldson}
in Section \ref{sec:virt small maps}.
First, we will introduce the moduli space $\Mf_\ga$ of stable framed vector bundles.
Under appropriate assumptions, this moduli space is smooth and
the canonical morphism $\pi:\Mf_\ga\to\M_\ga$ is projective.
The fibres of this map over $\Ms_\ga$ are projective spaces,
while the other fibres can be identified with moduli spaces of
stable nilpotent quiver representations (see Theorem \ref {th:fiber}).
This analysis of the fibres will be used in Theorem \ref{framed-semism} to show that $\pi$ is a virtually small map,
which is a generalization of the notion of a small map
(see \S\ref{sec:virt small}).
The properties of virtually small maps (see Theorem \ref{th:LT})
imply that the leading term of $\pi_*\ICV_{\Mf_\ga}$ (\wrt the weight degree on the Grothendieck group of mixed Hodge modules)
can be related to $\ICV_{\M_\ga}$.
On the other hand, we can express the class of $\pi_*\ICV_{\Mf_\ga}$
in terms of DT classes (see Theorem \ref{th:dt def2})
and we show that the leading term in this expression is related to $\DT_\ga$ in Theorem \ref{main:proof}.
While our main focus is on the moduli spaces of vector bundles on a curve, all our proofs generalize verbatim to moduli spaces in other hereditary categories.

\medskip

\subsection*{Acknowledgments}
The first author would like to thank Ben Davison, Jan Manschot
and Andr\'as Szenes for many useful discussions.
Both authors would like to thank J\"org Sch\"urmann for helpful remarks on mixed Hodge modules.

\subsection*{Disclaimer}
In 2025 a particular case of our result was proved in \cite{felisetti_parabolic}.
More precisely, the authors considered the projection $\pi:\Mf\to\M$,
where $\M=\bigsqcup_{r\ge0}\M_{r,0}$ and
$\Mf=\bigsqcup_{r\ge0}\Mf_{r,0}$ is the moduli space of stable framed objects for the parabolic framing functor \eqref{parab}
(one denotes $\cM^\f_{r,0}$ by $\cP_0(r)$ in \cite{felisetti_parabolic}) and proved a specialization of Theorem
\ref{main:proof} (see also Theorem \ref{th:dt def2}) to Poincar\'e polynomials
(see \cite[Theorem 8.7]{felisetti_parabolic}).
As in our paper, the authors perform a careful analysis of the map $\pi$ and its fibres in order to determine the pushforward of the intersection complex.

\section{Virtually small maps}
\label{sec:virt small maps}

\subsection{Mixed t-categories}
For a subcategory $\cS\sbe\cD$ (or a collection of subcategories) of a triangulated category $\cD$,
let $\ang{\cS}\sbe\cD$ be the minimal full subcategory that contains $\cS\cup\set0$ and is closed under extensions.
A t-structure $(\cD^{\le0},\cD^{\ge0})$ \BBD on a triangulated category $\cD$ is called bounded if
$\cD=\angs{\cA[n]}{n\in\bZ}$
for the heart $\cA=\cD^{\le0}\cap\cD^{\ge0}$.
For $M\in\cD$, we define the cohomology object
$H^iM=\ta_{\le0}\ta_{\ge0}(M[i])\in\cA$.
A (bounded) t-category is a triangulated category equipped with a (bounded) t-structure.

\begin{definition}
A bounded t-structure on a triangulated category $\cD$ is called mixed if its heart~$\cA$ is equipped with strictly full subcategories
$\cA_n\sbe\cA$, for $n\in\bZ$, satisfying
\begin{enumerate}
\item $\Hom^i(\cA_m,\cA_n)=0$ for $m<n+i$.
\item $\cA=\angs{\cA_n}{n\in\bZ}$.
\end{enumerate}
Objects of $\cA_n$ are called pure of weight $n$.
An object $M\in\cD$ is called pure of weight $n$
if $H^iM$ is pure of weight $n+i$ for all $i\in\bZ$.
A triangulated category $\cD$ equipped with a mixed t-structure is called a mixed t-category.
An abelian category $\cA$ such that the standard t-structure of $D^b(\cA)$ is mixed is called an (abelian) mixed category.
\end{definition}

\begin{remark}
The above notion is closely related to
ft-categories \cite{beilinson_koszula},
mixed abelian categories~\cite{beilinson_koszul},
weight structures \cite{bondarko_weight},
co-t-structures \cite{pauksztello_compact}.
\end{remark}

\begin{remark}
A structure as above will be called weakly mixed if the first axiom is substituted by the requirement $\Hom^i(\cA_m,\cA_n)=0$ for $m<n+i$, $m\ne n$ and $\Hom^i(\cA_n,\cA_n)=0$ for $i\ge2$.
\end{remark}

\begin{example}
For an algebraic variety $X$ over \bC, the abelian category $\MHM(X)$
of mixed Hodge modules is mixed and of finite length.
The category of pure weight $n$ objects of $\MHM(X)$ is the category
$\pHM(X,n)$ of polarizable pure Hodge modules of weight $n$ over $X$.
We have $\pHM(\pt, n)=\pHS(\bQ,n)$,
the category of polarizable $\bQ$-Hodge structures of weight $n$,
and $\MHM(\pt)=\MHS^p(\bQ)$,
the category of (graded-)polarizable mixed $\bQ$-Hodge structures.
\end{example}

\begin{example}
Let $X_0$ be an algebraic variety over a finite field.
Then the category
$D^b_m(X_0,\bar\bQ_\ell)\sbs D^b_c(X_0,\bar\bQ_\ell)$ of mixed complexes \BBD[\S5.1] is a weakly mixed category.
More precisely, it has a bounded (perverse) t-structure with the heart $\cA$
of finite length equipped with a weight filtration.
For any simple, non-isomorphic objects $M,N\in\cA$ one can prove that
$\Hom^i(M,N)=0$ if $w(M)<w(N)+i$
(the vanishing is proved for $w(M)+1<w(N)+i$ in \BBD and a similar argument can be used for our slightly stronger statement).
One can also prove that $\Hom^i(M,M)=0$ for $i\ge2$.
But $\Hom^1(M,M)$ is always nonzero for simple $M$.
While most of the results in this paper are formulated in terms of mixed Hodge modules, they can be similarly formulated in the context of $\ell$-adic sheaves.
\end{example}

\begin{lemma}
Let $\cA$ be an abelian category and $\cA_n\sbe\cA$, $n\in\bZ$,
be strictly full subcategories such that
$\Hom(\cA_m,\cA_n)=\Ext^1(\cA_m,\cA_n)=0$ for $m<n$
and $\cA=\angs{\cA_n}{n\in\bZ}$.
Then $\cA_n\sbe\cA$ are Serre subcategories of $\cA$
(closed under taking extensions, quotients, subobjects).
Every object $M\in \cA$ has a unique increasing filtration
$W_\bul$ (called the weight filtration) such that
\begin{enumerate}
\item $W_{-n}=0$ and $W_n=M$ for $n\gg0$.
\item $\Gr^W_n M=W_n/W_{n-1}\in\cA_n$ for all $n\in\bZ$.
\end{enumerate}
\end{lemma}
\begin{proof}
Left to the reader.
\end{proof}

\begin{lemma}[Decomposition theorem]
For every pure object $M\in\cD$
there is a (non-canonical) isomorphism $M\iso\bop_{i\in\bZ}H^i(M)[-i]$.
If $\cA_n$ is of finite length, then it is semisimple.
\end{lemma}
\begin{proof}
Left to the reader.
\end{proof}

For $M\in\cA$, we write
$w(M)\le n$ if $\Gr^W_k=0$ for $k>n$
and we write $w(M)\ge n$ if $\Gr^W_kM=0$ for $k<n$.
We define
\[w(M)=\sup\sets{n\in\bZ}{\Gr^W_nM\ne 0}\in\bZ\sqcup\set{-\infty}.\]
For $M\in\cD$, we write $w(M)\le n$ if $w(H^iM)\le n+i$ for all $i\in\bZ$ and we write $w(M)\ge n$ if $w(H^iM)\ge n+i$ for all $i\in\bZ$.
We define
\[w(M)=\max\sets{w(H^iM)-i}{i\in\bZ}\in\bZ\sqcup\set{-\infty}.\]
Let
\begin{gather*}
\cD_{\le n}=\sets{M\in\cD}{w(M)\le n}=\angs{\cA_m[i]}{m+i\le n},\\
\cD_{\ge n}=\sets{M\in\cD}{w(M)\ge n}=\angs{\cA_m[i]}{m+i\ge n}.
\end{gather*}
Then $\cD_{\le n}[1]=\cD_{\le n+1}$,
$\cD_{\ge n}[1]=\cD_{\ge n+1}$ and
\[\Hom(\cD_{\le m},\cD_{\ge n})=0\qquad \forall m<n.\]

\subsection{Exactness}
An additive functor $F:\cD_1\to\cD_2$ between triangulated categories is called exact (or triangulated) if it commutes with translations and preserves distinguished triangles.
For an exact functor $F:\cD_1\to\cD_2$ between t-categories
with the hearts $\cA_1,\cA_2$, we define
\begin{equation}
\p F:\cA_1\to\cA_2,\qquad M\mto H^0FM.
\end{equation}
An exact functor $F:\cD_1\to\cD_2$ between t-categories
is called right (resp.\ left) t-exact if $F(\cD_1^{\le0})\sbe \cD_2^{\le0}$
(resp.\ $F(\cD_1^{\ge0})\sbe \cD_2^{\ge0}$).
We say that $F$ has t-amplitude $[a,b]$ if
$F(\cD_1^{\le0})\sbe \cD_2^{\le b}$
and $F(\cD_1^{\ge0})\sbe \cD_2^{\ge a}$.

\begin{theorem}[See \BBD]
Let $f:X\to Y$ be a morphism between algebraic varieties
with fibres of dimension $\le d$.
Then (for the perverse t-structure on $D^b_c(\bQ_X)$ or the standard t-structure on $D^b\MHM(X)$)
\begin{enumerate}
\item $f_!,f^*$ have t-amplitude $\le d$,
meaning that $f_![d],f^*[d]$ are right t-exact.
\item $f_*,f^!$ have t-amplitude $\ge -d$,
meaning that $f_*[-d],f^![-d]$ are left t-exact.
\end{enumerate}
If $f$ is affine, then $f_*$ is right t-exact
and $f_!$ is left t-exact.
\end{theorem}

An exact functor $F:\cD_1\to\cD_2$ between mixed t-categories
is called right (resp.\ left) w-exact if
$F(\cD_{1,\le0})\sbe\cD_{2,\le0}$
(resp.~$F(\cD_{1,\ge0})\sbe\cD_{2,\ge0}$).
The following result is Deligne's generalization of Weil's conjectures in the $\ell$-adic context
\cite{deligne_la}
and Saito's theorem in the \MHM context
\cite{saito_mixed}.

\begin{theorem}[See {\cite{deligne_la,saito_mixed}}]
Let $f:X\to Y$ be a morphism between algebraic varieties.
Then
\begin{enumerate}
\item $f_!,f^*$ are right w-exact.
\item $f_*,f^!$ are left w-exact.
\end{enumerate}
\end{theorem}

\begin{remark}\label{rm:exact ts}
The external tensor product \cite{saito_mixed}
\[\boxtimes:\MHM(X)\xx \MHM(Y)\to \MHM(X\xx Y)\]
is w-exact, meaning that if $M\in\MHM(X)$ and $N\in\MHM(Y)$ are pure, then $M\boxtimes N$ is pure of weight $w(M)+w(N)$.
The induced external tensor product on the derived categories
is automatically t-exact (and w-exact).
On the other hand, the tensor product $\ts^*$ on $D^b\MHM(X)$, defined by $M\ts^* N=\De^*(M\boxtimes N)$ for the diagonal $\De:X\to X\xx X$,
is right t-exact and right w-exact (\cf \BBD[5.1.14]).
The duality functor $\bD:\MHM(X)\to\MHM(X)^\op$ maps objects of weight $n$ to objects of weight $-n$ (actually $\bD M\iso M(n)$ for any pure object $M$ of weight $n$).
\end{remark}

\subsection{Mixed Hodge modules}
\label{sec:MHM}
For an algebraic variety $X$ over $\bC$, let $d_X=\dim X$ denote the maximal dimension of its irreducible components.
Let $\MHM(X)$ be the category of mixed Hodge modules over $X$.
Let $D^b_c(\bQ_X)$ be the derived category of bounded $\bQ$-complexes with constructible cohomology and let $\Perv(\bQ_X)\sbs D^b_c(\bQ_X)$ be the subcategory of perverse sheaves.
The exact functor
$\rat:D^b\MHM(X)\to D^b_c(\bQ_X)$
maps $\MHM(X)$ to $\Perv(\bQ_X)$ and is compatible with Grothendieck's six operations on both sides.
We have $\MHM(\pt)=\MHS^p(\bQ)$,
the category of (graded-)polarizable mixed $\bQ$-Hodge structures.

Let $\bQ(n)\in\MHM(\pt)$ be the \idef{Tate object} of weight $-2n$ for $n\in\bZ$.
It induces the \idef{Tate twist} functor
$\MHM(X)\to\MHM(X)$, $M\mto M(n)=M\ts \bQ(n)$.
We consider the \idef{Lefschetz object}
\begin{equation}\label{lef1}
\bL=H^*_c(\bA^1,\bQ)=\bQ(-1)[-2]\in D^b\MHM(\pt)
\end{equation}
of weight zero and the corresponding functor
\begin{equation}
\bL:D^b\MHM(X)\to D^b\MHM(X),\qquad
M\mto M(-1)[-2].
\end{equation}
Note that $H^*_c(\bP^n,\bQ)=\bQ\oplus \bL\oplus\dots\oplus\bL^n$ for $n\ge0$.

In what follows, we will require the class $\bL^{1/2}$ at the level of Grothendieck groups (see \S\ref{GG1}),
but it will be convenient to have a related object also at the level of categories.
We have two approaches to this problem.
The first one is to embed $\MHM(X)$ into the category of monodromic mixed Hodge modules $\MMHM(X)$
(see \cite{kontsevich_cohomological,davison_cohomological}),
so that one has the objects $\bQ(1/2)\in\MMHM(\pt)$ and $\bL^{1/2}=\bQ(-1/2)[-1]\in D^b\MMHM(\pt)$ and the corresponding twist functors.
The second approach is to construct the root category $\cA^{1/2}$
and the functor $\bT^{1/2}=\bQ(1/2):\cA^{1/2}\to\cA^{1/2}$
for the category $\cA=\MHM(X)$ and the functor $\bT=\bQ(1):\cA\to\cA$.
More generally, we construct the root category $\cA^{1/r}$ and the functor $\bT^{1/r}:\cA^{1/r}\to\cA^{1/r}$ for any $r\ge1$ as follows.
We define $\cA^{1/r}=\bop_{i=0}^{r-1}\cA^{(i)}$ with $\cA^{{(i)}}=\cA$ and the functor $\bT^{1/r}$ sending
$\cA^{(i)}\ni M\mto M\in\cA^{(i+1)}$ for $0\le i<r-1$
and $\cA^{(r-1)}\ni M\mto\bT M\in\cA^{(0)}$.
We embed $\cA=\cA^{(0)}\emb\cA^{1/r}$ so that $(\bT^{1/r})^r=\bT$.
The weight of $M\in\cA^{(i)}\sbs\cA^{1/2}$ is defined to be $w(M)-i$ (if $M\in\cA$ is pure) so that $\bT^{1/2}$ has weight $-1$ and $\bT=\bQ(1)$ has weight $-2$ as before.
We can also define the functor $\bQ(n/2)$ for any $n\in\bZ$.
The same construction can be applied to $D^b\MHM(X)$ and we define
$\bL^{1/2}=\bQ(-1/2)[-1]$.
By abuse of notation we will continue to write $\MHM(X)$ instead of $\MHM(X)^{1/2}$.
We define
\begin{equation}
M\ang n=M(n/2)[n]=\bL^{-n/2}M,\qquad M\in D^b\MHM(X).
\end{equation}
Note that the weight of $M\ang n$ is equal to the weight of $M$.

Let $X$ be an irreducible algebraic variety of dimension $d$.
The object
\[\bQ_X=a_X^*\bQ\in D^b\MHM(X),\qquad a_X:X\to\pt,\]
has weight $\le0$.
If $X$ is smooth, then $\bQ_X$ has weight 0 and $\bQ_X[d]\in\MHM(X)$.
Generally, we define the intersection complex
\begin{equation}
\ICV_X
=j_{!*}(\bQ_U\ang d)\in\MHM(X),
\end{equation}
where $j:U\emb X$ is an embedding of a nonempty, open, smooth subvariety and
\begin{equation}
j_{!*}(M)=\Im(\p j_! M\to \p j_* M)\in\MHM(X),\qquad M\in\MHM(U).
\end{equation}
The object $\ICV_X$ is self-dual (meaning that $\bD\ICV_X\iso\ICV_X$)
and has weight zero.
The object $\IC_X=\rat(\ICV_X)$
is the perverse intersection complex in $\Perv(\bQ_X)$.
We define the intersection cohomology
\begin{equation}\label{i-coh2}
\IH^*(X,\bQ)
=H^*(X,\ICV_X\ang{-d})
\end{equation}
so that $\IH^*(X,\bQ)=H^*(X,\bQ)$ for smooth $X$.

\begin{remark}
Some authors define the intersection complex
$\ICV'_X=j_{!*}(\bQ_U[d])\in\MHM(X)$.
It is pure of weight $d$ and satisfies $\bD\ICV'_X\iso\ICV'_X(d)$ (the same is true for any pure object of weight $d$).
One has $\ICV'_X=\Gr^W_d H^d\bQ_X$ \sai,
where $H^d$ is the $d$-th cohomology in $D^b\MHM(X)$.
\end{remark}

\subsection{Degrees}
\label{GG1}
For a graded abelian group $V=\bop_{i\in \bZ}V_i$
and $x=\sum_{i\in \bZ} x_i\in V$ with $x_i\in V_i$, let
\begin{equation}
\deg(x)=\sup\sets{i}{x_i\ne0}\in\bZ\sqcup\set{-\infty}.
\end{equation}
If $A$ is a graded integral domain and $V$ is a graded torsion-free $A$-module, then
\begin{equation}
\deg(ax)=\deg(a)+\deg(x)\qquad \forall a\in A,\, x\in V.
\end{equation}
We extend the degree map to $\cF(A)\ts_A V$,
where $\cF(A)$ is the fraction field of $A$,
by the formula
\begin{equation}
\deg(x/a)=\deg(x)-\deg(a)\qquad \forall a\in A\ms\set0,\, x\in V.
\end{equation}

For an abelian category $\cA$ (resp.~ a triangulated category $\cD$), let $K(\cA)$ (resp.~$K(\cD)$) denote its Grothendieck group.
Let $[M]\in K(\cD)$ denote the class of $M\in\cD$.
If $\cD$ is a bounded t-category with the heart~$\cA$, then
the canonical map $K(\cA)\to K(\cD)$ is an isomorphism.
If $\cD$ is a mixed t-category with the heart $\cA$,
then the Grothendieck group $K(\cA)\iso K(\cD)$ is graded,
where we define $\deg[M]=n$ for a pure object $M\in\cA$ of weight~$n$.
If~$\cA$ is of finite length and $C_n$ denotes
the set of isomorphism classes of simple objects in $\cA_n$, then $K(\cA)\iso\bop_{n\in\bZ}\bZ C_n$.

In particular, the Grothendieck group
$K(\MHM(X))\iso K(D^b\MHM(X))$
is a graded module over the graded ring $K(\MHM(\pt))$ (\cf \S\ref{sec:la-MHM}).
Let
\begin{gather}
\bL=[H_c^*(\bA^1,\bQ)]=[\bQ(-1)]\in K(\MHM(\pt))\label{lef2}\\
\K(\MHM(X))=K(\MHM(X))\ts_{\bZ[\bL]}\bZ[\bL^{\pm1/2},(1-\bL^n)\inv\mid n\ge1].\label{ext K1}
\end{gather}
Note that $\bL$ has degree $2$.
We extend the degree function to $\K(\MHM(X))$ so that $\bL^{1/2}$ has degree $1$.
For $M\in D^b(\MHM(X))$, let $\deg (M)=\deg[M]$.
For example, $\deg(\bQ[n])=0$, while $w(\bQ[n])=n$ for $n\in\bZ$.
The duality functor $\bD$ on $\MHM(X)$ induces the group homomorphism
$\bD:\K(\MHM(X))\to \K(\MHM(X))$ such that $\bD(\bL^n a)=\bL^{-n}\bD(a)$ for $n\in\oh\bZ$ and $a\in\K(\MHM(X))$.
The class $[\ICV_X]\in\K(\MHM(X))$ is self-dual
(meaning that $\bD[\ICV_X]=[\ICV_X]$).

\subsection{Decompositions}
We say that an object $M\in\MHM(X)$
is supported on a locally closed subvariety $Z\sbe X$
if $M=j_{!*}\p j^*M$ for the embedding $j:Z\emb X$,
meaning that the canonical map $M\to \p j_*\p j^*M$ induces an isomorphism $M\isoto j_{!*}\p j^*M\emb \p j_*\p j^*M$.
For a simple object $M$, this means that $Z$ contains a nonempty open subset of $\supp M$.
We say that $M\in D^b\MHM(X)$ is supported on
$Z$ if every $H^iM$ is supported on $Z$.
In what follows a partition of $X$ means a finite collection $\cS$ of locally closed subsets of $X$ such that $X=\bigsqcup_{S\in\cS}S$.

\begin{lemma}\label{lm:decomp1}
Let $\cS$ be a partition of an algebraic variety $X$
and $M\in\MHM(X)$ be a pure object.
Then there is a unique decomposition
$M=\bop_{S\in\cS}M_S$, where $M_S$ is supported on $S$.
We have $M_S=j_{S!*}\p j_S^*M$.
\end{lemma}
\begin{proof}
We can assume that $M$ is simple.
There exists a smooth, irreducible, locally closed $U\sbs X$ such that
$M=j_{!*}\p j^*M$ for $j:U\emb X$.
Let $S\in\cS$ be such that $d_{S\cap U}$ is maximal.
Then $S\cap U$ is open in $U$.
We can substitute $U$ by $U\cap S$ and assume that $U\sbs S$.
Then $M=j_{S!*}\p j_S^*M$.

To prove uniqueness, we need to show that
$\p j^*j_{!*}=\id$ and $\p i^*j_{!*}=0$
for two embeddings $j:S\emb X$ and $i:T\emb X$ such that $S\cap T=\es$.
The first equation is clear.
As $j_{!*}N=\Im(\p j_{!}N\to \p j_{*}N)$ and
$\p i^*$ is right exact, it is enough to show that
$\p i^*\p j_{!}=0$.
We have
$\p i^*\p j_{!}=\ta_{\ge0}i^*\ta_{\ge0}j_{!}
=\ta_{\ge0}i^*j_{!}=0$ as
$i^*j_{!}=0$ and $\ta_{\ge0}i^*\ta_{\ge0}=\ta_{\ge0}i^*$.
\end{proof}

\begin{lemma}\label{lm:decomp2}
Let $\cS$ be a partition of an algebraic variety $X$
and $M\in D^b\MHM(X)$ be a pure object.
Then there is a decomposition
$M=\bop_{S\in\cS}M_S$,
where $M_S$ is supported on $S$.
We have $M_S\iso \bop_i (j_{S!*}\p j_S^*H^iM)[-i]$.
\end{lemma}
\begin{proof}
Let $e_S=j_{S!*}\p j^*_S$ and
$\bar e_S(M)=\bop_i (e_S H^iM)[-i]$ for $S\in\cS$.
We have $e_S^2=e_S$ and $e_Se_T=0$ for $S\ne T$.
Therefore $\bar e_S^2=\bar e_S$ and $\bar e_S\bar e_T=0$ for $S\ne T$.
Using $M\iso\bop_i (H^iM)[-i]$
and applying the previous result to every $H^iM$, we obtain
$M\iso\bop_S \bar e_S(M)$.
Let $M=\bop_S M_S$ be another decomposition with $M_S$ supported on $S$.
We have $\bar e_S (M_T)=0$ for $S\ne T$, hence
$\bar e_S(M)=\bar e_S(M_S)$.
By the same argument as before, we have $M_S\iso\bop_T \bar e_T(M_S)=\bar e_S(M_S)=\bar e_S(M)$.
\end{proof}

\subsection{Defects}
\label{sec:virt small}
Let $\pi:X\to Y$ be a projective morphism between algebraic varieties.
Let $Y_i=\sets{y\in Y}{\dim \pi\inv(y)=i}$ for $i\in\bN\sqcup\set{-\infty}$ (note that $\dim\es=-\infty$).
The subsets $Y_{\le k}=\bigsqcup_{i\le k}Y_i$ are open in $Y$.
For a subvariety $Z\sbe Y$, we define the (fibre) defect
(\cf~\cite{goresky_stratified,cataldo_hodgea})
\begin{equation}\label{def1}
\de(\pi,Z)=\max\sets{2i+d_{Z\cap Y_i}-d_X}{i\ge0}.
\end{equation}
If $\cS$ is a finite partition of~$Y$
such that the fibres of $\pi$ over $S\in\cS$ have dimension $c_S$, then
\begin{equation}
\de(\pi,Z)=\max_{S\in\cS}{\de(\pi,Z\cap S)},\qquad
\de(\pi,Z\cap S)=2c_S+d_{Z\cap S}-d_X.
\end{equation}

Note that $\de(\pi)=\de(\pi,Y)=d_{X\xx_YX}-d_X\ge0$.
A (surjective) morphism $\pi$ is called semismall if $\de(\pi)=0$ and  small if $\de(\pi)=0$ and $\de(\pi,Y_i)=0$ only for $i=0$.

\begin{definition}\label{def:v-small}
A projective morphism $\pi:X\to Y$ is called $n$-small
(or virtually small if $n$ is clear from the context)
if there exists an open subset $U\sbe Y$
such that $\de(\pi,Y\ms U)<n$ and $U\sbe Y_n$.
Note that $\de(\pi)\le n$ if $\pi$ is $n$-small.
\end{definition}

\begin{remark}
If $\de(\pi)<n$, we can assume that the open set $U$ above is empty.
If $\de(\pi)=n$, we can assume that $U$ is smooth and equidimensional (of dimension $d_X-n$).
If $X$ is smooth, we can assume that $\pi$ is smooth over $U$.
\end{remark}

\begin{lemma}
If $\pi:X\to Y$ is a projective surjective morphism and $X$, $Y$ are irreducible, then $\pi:X\to Y$ is $n$-small if and only if $\de(\pi,Y_i)<n$ for $i\ne n$ and $\de(\pi,Y_n)\le n$.
If $\de(\pi,Y_n)=n$, then $Y_n\sbe Y$ is open and $n=d_X-d_Y$.
\end{lemma}
\begin{proof}
Let $\pi$ be $n$-small and $U\sbe Y_n$ be such that $\de(\pi,Y\ms U)<n$.
Then $Y_i\sbe Y\ms U$ for $i\ne n$, hence $\de(\pi,Y_i)<n$.
We also have $\de(\pi,Y_n)\le\de(\pi)\le n$.
Conversely, assume that $\de(\pi,Y_i)<n$ for $i\ne n$ and $\de(\pi,Y_n)\le n$.
If $\de(\pi,Y_n)<n$, then $\de(\pi)<n$ and we can take $U=\es$.
If $\de(\pi,Y_n)=n$, then $d_X=n+d_{Y_n}=d_{\pi\inv(Y_n)}$, hence $\pi\inv(Y_n)\sbe X$ is open.
The subset $\pi\inv(Y_{<n})$ is also open, hence $\pi\inv(Y_{<n})=\es$ and $Y_{<n}=\es$.
Therefore $Y_n\sbe Y$ is open and we can take $U=Y_n$.
We have $d_X=n+d_{Y_n}=n+d_Y$, hence $n=d_X-d_Y$.
\end{proof}

In what follows, let $\cD$ denote the t-category $D^b\MHM(X)$
with the standard t-structure
or the t-category $D^b_c(\bQ_X)$ with the perverse t-structure.
These t-structures are compatible under the functor
$\rat:D^b\MHM(X)\to D^b_c(\bQ_X)$.
For $K\in D^b_c(\bQ_X)$, let $\sH^iK\in\Sh(\bQ_X)$ denote the $i$-th cohomology sheaf and $\p H^iK\in\Perv(\bQ_X)$ denote the $i$-th cohomology object \wrt the perverse t-structure.

\begin{lemma}\label{strat bound}
Let $M\in \cD=D^b\MHM (Y)$ and $\cS$ be a partition of $Y$
such that $j_S^*M\in \cD^{\le n}$ for the embeddings $j_S:S\emb Y$, $S\in\cS$.
Then $M\in\cD^{\le n}$.
\end{lemma}
\begin{proof}
Let $U\sbe Y$ be open and $j:U\emb Y$, $i:Y\ms U\emb U$ be the corresponding embeddings.
For every $M\in D^b(\MHM Y)$ there is a triangle
$j_!j^*M\to M\to i_!i^*M\to$.
If $i^*M\in\cD^{\le n}$ and $j^*M\in\cD^{\le n}$, then $M\in\cD^{\le n}$ as $i_!,j_!$ are right t-exact.
Given a partition, we can refine it so that it satisfies the
frontier condition: if $S\cap\bar T\ne\es$, then $S\sbe\bar T$ for $S,T\in\cS$.
Note that we still have $j^*_S M\in\cD^{\le n}$ for all $S$ as the restriction functor is right exact.
We can apply the previous argument inductively and deduce that $M\in\cD^{\le n}$.
\end{proof}

\begin{lemma}
\label{defect bound}
Let $\pi:X\to Y$ be a projective morphism such that $X$ is smooth.
For a locally closed embedding $j:Z\emb Y$
we have $j^*\pi_*\ICV_X\in\cD^{\le \de(\pi,Z)}$.
\end{lemma}
\begin{proof}
By Lemma \ref{strat bound} we can assume that $Z\sbs Y_k$ for some $k\ge0$.
Consider the Cartesian square
\[\begin{tikzcd}
X'\rar["j'"]\dar["\pi'"']&X\dar["\pi"]\\
Z\rar["j"]&Y
\end{tikzcd}\]
An object $K\in D^b_c(\bQ_X)$ is contained in $\p D^{\le n}$ \iff $\dim(\supp \sH^iK)\le -i+n$ for all $i\in\bZ$.
We have
$j^*\pi_*\IC_X=\pi'_*j'^*\bQ_X[d_X]$.
The object $K=j'^*\bQ_X[d_X]$ satisfies
$\dim(\supp \sH^iK)=d_{X'}=-i+(d_{X'}-d_X)$ for $i=-d_X$
and $\sH^i K=0$ otherwise.
Therefore $K\in\p D^{\le d_{X'}-d_X}$.
The functor $\pi'_*$ has amplitude $[-k,k]$.
Therefore
$j^*\pi_*\IC_X
=\pi'_*K
\in\p D^{\le k+d_{X'}-d_X}
=\p D^{\le\de(\pi,Z)}.$
\end{proof}

\begin{lemma}\label{lm:deg1}
Let $f:X\to Y$ be a morphism between algebraic varieties and
$M\in \cD=D^b\MHM (Y)$ be a pure object of weight $0$
such that $f^*M\in \cD^{\le n}$.
Then $\deg({\p f^*}H^iM)\le n$ for all $i\in\bZ$.
\end{lemma}
\begin{proof}
By our assumption $f^*(H^iM)[-i]\in\cD^{\le n}$,
hence $f^*H^iM\in\cD^{\le n-i}$.
If $\p f^*H^iM\ne0$, then $i\le n$.
As $H^iM$ has weight $i$, the object
$\p f^*H^iM$ has weight $\le i\le n$.
\end{proof}

\begin{lemma}\label{lm:deg2}
Let $\pi:X\to Y$ be a projective morphism with smooth $X$.
Let $\cS$ be a partition of $Y$ and $\pi_*\ICV_X=\bop_{S\in\cS} M_S$ be a decomposition such that $M_S$ is supported on $S$.
Then $\deg M_S\le\de(\pi,S)$ for all $S\in\cS$.
\end{lemma}
\begin{proof}
We have $M_S\iso\bop_i (j_{S!*}\p j^*_S H^iM)[-i]$ for $M=\pi_*\ICV_X$, hence we need to show that $\deg \p j^*_SH^iM\le\de(\pi,S)$.
The object $M$ is pure of weight zero and
by Lemma \ref{defect bound} we have $j^*_S M\in\cD^{\le\de(\pi,S)}$.
By Lemma \ref{lm:deg1} we obtain $\deg\p j_S^* H^iM\le\de(\pi,S)$.
\end{proof}

\subsection{The leading term}
\label{sec:LT}
If $\pi:X\to Y$ is a projective morphism and is a (topological) fibre bundle with all fibres irreducible and of dimension~$n$,
then (\cf \cite{gottsche_perverse})
\begin{equation}\label{top cohom}
H^{n}\pi_*(\bQ_X[d_X])\iso \bQ_Y[d_Y](-n).
\end{equation}
If $X$ and $Y$ are smooth (and irreducible), we obtain
(\cf Appendix \ref{RHL})
\begin{equation}
H^n \pi_*\ICV_X\iso \ICV_Y(-n/2).
\end{equation}

\begin{theorem}\label{th:LT}
Let $\pi:X\to Y$ be an $n$-small morphism with smooth $X$
and let $U\sbe Y$ be open smooth equidimensional such that $\de(\pi,Y\ms U)<n$, $\de(\pi,U)=n$ and the fibres of $\pi$ over $U$ are smooth connected of dimension $n$.
Then $\pi_*\ICV_X$ has degree $n$ and
the leading term $\ICV_{\ubar U}\ang{-n}$.
\end{theorem}
\begin{proof}
Let $j:U\emb Y$ and $i:Z=Y\ms U\emb Y$ be the embeddings.
For $M=\pi_*\ICV_X$, we have
$M=M_U\oplus M_Z$, where $M_U\iso \bop_k(j_{!*}j^*H^kM)[-k]$
and $M_Z\iso \bop_k(i_{!*}\p i^*H^kM)[-k]$.
We have
$\deg(M_Z)\le\de(\pi,Z)<n$
by Lemma \ref{lm:deg2}.
On the other hand $j^*H^kM=H^k j^*M=H^k \pi'_*\ICV_{X'}$ for $\pi':X'=\pi\inv(U)\to U$.
The object $H^k\pi'_*\ICV_{X'}$ is zero for $k>n$ and has degree $<n$ for $k<n$.
On the other hand $H^n\pi'_*\ICV_{X'}=\ICV_U(-n/2)$,
hence the leading term of $M_U$ is
$(j_{!*}\ICV_U(-n/2))[-n]=\ICV_{\ubar U}(-n/2)[-n]$.
\end{proof}

We prove the following well-known result
(\cf \cite[Theorem 5]{gottsche_perverse}) for completeness.

\begin{theorem}
Let $\pi:X\to Y$ be a semismall morphism with smooth $X$.
Let $\cS$ be a smooth partition of $Y$ such that $\pi$ is a (topological) fibre bundle with irreducible fibres over every $S\in\cS$.
Then $\pi_*\ICV_X=\bop_{\de(\pi,S)=0}\ICV_{\bar S}$.
\end{theorem}
\begin{proof}
We have $M=\pi_*\ICV_X\in\cD^{\le0}$ by Lemma \ref{defect bound}.
As $M$ is self-dual, we conclude that $M\in\MHM(X)$.
By Lemma \ref{lm:decomp1} we have $M=\bop_S \p j_{S!*}\p j^*_S M$.
If $\de(\pi,S)<0$, then $j_S^*M\in\cD^{<0}$, hence $\p j_S^*M=0$.
Let us assume that $\de(\pi,S)=0$ and let $n=d_X-d_{X'}$ be the fibre dimension, where $X'=\pi\inv(S)$.
Then
\[
\p j_S^*M
=H^0(\pi'_{*}\bQ_{X'}[d_X](d_X/2))
=H^{n}(\pi'_{*}\bQ_{X'}[d_{X'}])(d_X/2)
\iso\bQ_S[d_S](d_X/2-n)=\ICV_S,
\]
where we applied \eqref{top cohom} to $\pi':X'\to S$.
\end{proof}

\section{\tpdf{\la}{lambda}-rings over commutative monoids}
\label{sec:monoids}

\subsection{Pre-\tpdf{\la}{lambda}-ring of motivic functions}
Let $\Sch=\Sch_K$ be the category of algebraic schemes over a field $K$ of characteristic zero, meaning schemes $X$ equipped with a finite type morphism $a_X:X\to\pt=\Spec(K)$.
For a scheme (or an Artin stack) $S$ locally of finite type over $K$,
let $\Sch\qt S$ be the category of morphisms $X\to S$ over $K$, where $X\in\Sch$.
For a morphism $f:S\to T$ we have the functor
\[f_!:\Sch\qt S\to\Sch\qt T,\qquad [X\to S]\mto[X\to S\to T].\]
If $f:S\to T$ is of finite type, then $f_!$ has the right adjoint functor
\[f^*:\Sch\qt T\to\Sch\qt S,\qquad [X\to T]\mto[X\xx_T S\to S].\]
We also have the exterior product
\begin{equation*}
\boxtimes:\Sch\qt S\xx\Sch\qt T\to\Sch\qt (S\xx T),\qquad
[X\to S]\boxtimes[Y\to T]=[X\xx Y\to S\xx T].
\end{equation*}

Let $K(\Sch\qt S)$ be the Grothendieck group of algebraic schemes over $S$ (see \eg \cite{joyce_motivic,bridgeland_introduction}).
The above functors induce morphisms between Grothendieck groups.
The Grothendieck group $K(\Sch)=K(\Sch\qt\pt)$ has a commutative ring structure defined by $[X]\cdot [Y]=[X\xx Y]$.
The Grothendieck group $K(\Sch\qt S)$ is a module over $K(\Sch)$.
We define (\cf \eqref{ext K1})
\begin{gather}
\bL=[\bA^1]\in K(\Sch),
\label{lef3}\\
\K(\Sch\qt S)
=K(\Sch\qt S)\ts_{\bZ[\bL]}\bZ[\bL^{\pm1/2},(1-\bL^n)\inv\col n\ge1].
\label{ext K2}
\end{gather}
Consider a morphism $f:\cX\to S$, where $\cX$ is a finite type Artin stack with affine stabilizers.
Then there exists an algebraic variety $X$ with an action of the group $\GL_n$ and a geometric bijection $X\qt\GL_n\to \cX$
(see \eg \cite{bridgeland_introduction}).
The class
\begin{equation}
[\cX\to S]=[X\to\cX\to S]/[\GL_n]\in\K(\Sch\qt S)
\end{equation}
depends only on the morphism $f$ (see \eg \cite{bridgeland_introduction}).

Let $(S,\mu,\eta)$ be a commutative monoid in the category of schemes (or Artin stacks) over $K$.
We equip $\Sch\qt S$ with a symmetric monoidal structure
having the tensor product
\begin{equation}\label{ts:Sch}
[X\xto fS]\ts[Y\xto gS]=\mu_!(f\boxtimes g)
=[X\xx Y\to S\xx S\xto\mu S]
\end{equation}
and the unit object $\one=[\pt\xto\eta S]$.
This tensor product induces the ring structure on
$K(\Sch\qt S)$.
We equip $K(\Sch\qt S)$ with a pre-\la-ring structure having \si-operations
\begin{equation}
\si^n[X\to S]=[S^n(X)\to S^n(S)\xto\mu S],\qquad n\ge1,
\end{equation}
where $S^n(X)=X^n/S_n$.
We extend the pre-\la-ring structure to $\K(\Sch\qt S)$ by the formula
\begin{equation}\label{ext la-str}
\si^n(y^m a)=y^{mn}\si^n(a),\qquad
y=-\bL^{1/2},\,a\in\K(\Sch\qt S)
\end{equation}
In particular, $\pt$ is a commutative monoid, hence $K(\Sch)$ and $\K(\Sch)$ are pre-\la-rings.
The maps $\eta_!:K(\Sch)\to K(\Sch\qt S)$
and $a_{S!}:K(\Sch\qt S)\to K(\Sch)$ (induced by the unit $\eta:\pt\to S$ and the projection $a_S:S\to\pt$) are homomorphisms of pre-\la-rings .

\subsection{Graded pre-\tpdf{\la}{lambda}-rings and completions}
\label{sec:gr-pre-la}
More generally, let $\La$ be a commutative monoid and $S=\bigsqcup_{\ga\in\La}S_\ga$ be a $\La$-graded commutative monoid in the category of schemes over $K$.
The pre-\la-rings $K(\Sch\qt S)=\bop_{\ga\in\La}K(\Sch\qt S_\ga)$
and $\K(\Sch\qt S)=\bop_{\ga\in\La}\K(\Sch\qt S_\ga)$
are $\La$-graded pre-\la-rings, meaning that
\begin{equation}\label{gr la-ring}
\deg(ab)=\deg(a)+\deg(b),\qquad
\deg\si^n(a)=n\deg a.
\end{equation}
In particular, we consider the commutative monoid
$\bLa=\bigsqcup_{\ga\in\La}\pt_\ga$ with $\pt_\ga=\Spec(K)$ in the category of schemes over $K$.
Then
$K(\Sch\qt\bLa)=K(\Sch)[\La]=\bop_{\ga\in\La}K(\Sch)t^\ga$ is a
\La-graded pre-\la-ring with the product and \si-operations
\[[X]t^\ga\cdot [Y]t^{\ga'}=[X\xx Y]t^{\ga+\ga'},\qquad
\si^n([X]t^\ga)=[S^nX]t^{n\ga}.\]
The degree map
\begin{equation}\label{deg map}
\deg:S\to\bLa,\qquad S_\ga\ni x\mto\pt_\ga,
\end{equation}
is a homomorphism of commutative monoids
and induces a homomorphism of pre-\la-rings
$\deg_!:K(\Sch\qt S)\to K(\Sch\qt \bLa)=K(\Sch)[\La]$.

Assume that there exists a filtration
$\La=\La_0\spe\La_1\spe\dots$ such that
\begin{equation}\label{filt}
\La_i+\La_j\sbe\La_{i+j},\qquad
\n{\La\ms\La_i}<\infty,\qquad
\bigcap\nolimits_{i\ge0}\La_i=\es.
\end{equation}
For example, if $\La\sbe\bN^n$ for some $n\ge1$, we can define
$\La_k=\sets{\ga\in\La}{\sum_{i=1}^n\ga_i\ge k}$.
For a $\La$-graded ring $A=\bop_{\ga\in\La}A_\ga$, we define its completion
\begin{equation}\label{compl}
\what A=\ilim_k \rbr{A/\bop\nolimits_{\ga\in\La_k}A_\ga}
\iso\prod\nolimits_{\ga\in\La}A_\ga,
\end{equation}
where the last isomorphism is an isomorphism of abelian groups.
In particular, the completion $\what\K(\Sch\qt S)\iso\prod_{\ga\in\La}\K(\Sch\qt S_\ga)$
inherits the structure of a pre-\la-ring.
Let us assume that
$\La_1=\Lap=\La\ms\set0$ and consider the ideal
$\what\K_+(\Sch\qt S)=\prod_{\ga\in\Lap}\K(\Sch\qt S_\ga)$.
The plethystic exponential
\begin{equation}\label{Exp}
\Exp:\what\K_+(\Sch\qt S)\to 1+\what\K_+(\Sch\qt S),\qquad
a\mto \sum_{n\ge0}\si^n(a),
\end{equation}
where $\si^0(a)=1$,
is a group isomorphism (between the additive and the multiplicative groups).

\subsection{\tpdf{\la}{lambda}-ring of mixed Hodge modules}
\label{sec:la-MHM}
Let $(S,\mu,\eta)$ be a commutative monoid in the category of complex
algebraic varieties such that the product $\mu:S\xx S\to S$
is a finite map.
We will equip the category $\MHM(S)$ with a symmetric monoidal structure similarly to the case of $\Sch\qt S$ considered earlier (see
\eqref{ts:Sch}).

\begin{theorem}\label{th:sym monoidal}
The category $\MHM(S)$ with the tensor product
\[\ts:\MHM(S)\xx \MHM(S)\to \MHM(S),\qquad E\ts F=\mu_!(E\boxtimes F),\]
and the unit object $\one=\eta_!\bQ$
is a symmetric monoidal category.
The tensor product is exact and w-exact.
\end{theorem}
\begin{proof}
Let $\si:S\xx S\to S\xx S$ be the permutation of factors.
By \cite[Theorem 1.9]{maxim_symmetric}
(applied to the variety $X=S\sqcup S$),
there exists a canonical isomorphism
$$\si^\#:E\boxtimes F\to \si_*(F\boxtimes E)$$
satisfying $\si_*\si^\#\circ \si^\#=\Id$.
Applying the pushforward $\mu_!$ and using commutativity of the monoid,
we obtain an isomorphism $\si_{EF}:E\ts F\to F\ts E$ satisfying $\si_{EF}\circ\si_{FE}=\Id$.
This is the required braiding for the tensor product.
We have $\one\ts F
=\mu_!(\eta_!\bQ\boxtimes F)
=\mu_!(\eta\xx\id)_!(\bQ\boxtimes F)=F$.
The tensor product is exact (resp.\ w-exact) as both $\mu_!$ and $\boxtimes$ are exact (resp.\ w-exact).
\end{proof}

The above tensor product $\ts$ on $\MHM(S)$
should not be confused with the tensor product $\ts^*$ on $D^b\MHM(S)$ defined in Remark \ref{rm:exact ts},
which is only right w-exact.

\begin{remark}
Without the assumption that $\mu$ is finite, we can similarly equip the category $D^b\MHM(S)$ with the symmetric monoidal structure.
There is a monoidal functor
\begin{equation}\label{chi1}
\hi_c:\Sch\qt S\to D^b\MHM(S),\qquad
[X\xto fS]\mto f_!\bQ_X.
\end{equation}
\end{remark}

Let $\cA=\MHM(S)$ (or any other Karoubian symmetric monoidal category linear over $\bQ$).
By \cite{getzler_mixed,heinloth_note} the split Grothendieck group
$\ub K(\cA)$ (with relations induced by direct sums)
is a (special) \la-ring,
with the product and \si-operations defined by
\begin{gather}
[E]\cdot[F]=[E\ts F],\\
\si^n[E]=[S^n(E)],\qquad
S^n(E)=\Im\rbigg{\frac1{n!}\sum_{\si\in S_n}\si}\sbe E^{\ts n},
\end{gather}
where $e=\frac1{n!}\sum_{\si\in S_n}\si\in\End(E^{\ts n})$ is an idempotent and $\Im(e)$ is obtained by its splitting.

\begin{remark}
More generally, for any partition $\la\partof n$,
let $V_\la$ be the corresponding simple representation of the group $S_n$.
We define the Schur functor (\cf \cite{deligne_categories})
\begin{equation}
S_\la:\cA\to\cA,\qquad E\mto \Hom_{S_n}(V_\la,E^{\ts n}),
\end{equation}
where we first construct $\Hom(V_\la,E^{\ts n})$ as a direct sum of $\dim V_\la$ copies of $E^{\ts n}$ and then take the $S_n$-invariant subobject of $\Hom(V_\la,E^{\ts n})$ by splitting the idempotent $e=\frac1{n!}\sum_{\si\in S_n}\si$.
In particular, for the trivial representation we obtain
$S^n(E)=S_{(n)}(E)$
and for the alternating representation we obtain
\begin{equation}
\La^nE=S_{(1^n)}(E)=\Im\rbigg{\frac1{n!}\sum_{\si\in S_n}(-1)^{\si}\si}\sbe E^{\ts n}.
\end{equation}
We have $s_\la[E]=[S_\la(E)]$, where $s_\la$ is a symmetric Schur function.
In particular, $\la^n[E]=[\La^n(E)]$.
\end{remark}

\begin{theorem}\label{th:la-MHM}
The \la-ring structure on $\ub K(\MHM(S))$ descends to
the \la-ring structure on $K(\MHM(S))$ with the unit element $1=[\eta_*\bQ]$.
This \la-ring is $\bZ$-graded (see \eqref {gr la-ring}) by weight.
\end{theorem}
\begin{proof}
The above product on $\ub K(\MHM(S))$ descends to $K(\MHM(S))$ as the tensor product is exact.
The \la-structure descends to $K(\MHM(S))$
by \cf \cite[Lemma 2.1]{maxim_twisted} and
\cite[Lemma 4.1]{biglari_rings}.
We have seen in \S\ref{GG1} that $K(\MHM(S))$ is a $\bZ$-graded abelian group.
The product preserves degrees since the tensor product $\ts$ on $\MHM(S)$ is w-exact.
To see that the \la-structure respects the grading (see \eqref{gr la-ring}), we note that if $E\in\MHM(S)$ is pure of weight $m$, then $E^{\ts n}$ is pure of weight $mn$ and so is
$S^n(E)\sbe E^{\ts n}$.
\end{proof}

The monoidal functor defined in \eqref{chi1} induces a morphism of pre-\la-rings
(\cf~\cite{schuermann_characteristic,maxim_twisted})
\begin{equation}\label{chi2}
\hi_c:K(\Sch\qt S)\to K(\MHM(S)),\qquad
[X\xto fS]\mto[f_!\bQ_X].
\end{equation}
We extend the \la-ring structure to $\K(\MHM(S))$ using \eqref{ext la-str}.
Then $\hi_c$ extends to a morphism of pre-\la-rings
$\hi_c:\K(\Sch\qt S)\to \K(\MHM(S))$.

\begin{remark}
Let \cD be a (Karoubian) bounded t-category with a t-exact symmetric monoidal structure (meaning that it preserves the heart $\cA$).
Then the \la-ring structure on $\ub K(\cD)$ descends to the \la-ring structure on the Grothendieck group $K(\cD)$ (with relations induced by distinguished triangles, see \cite[\S4]{biglari_rings}).
The canonical isomorphism $K(\cA)\isoto K(\cD)$ is an isomorphism of \la-rings.
We can apply this statement to $\cD=D^b\MHM(S)$,
where $(S,\mu,\eta)$ is a commutative monoid in the category of complex algebraic varieties with finite maps $\mu,\eta$.
\end{remark}

Note that we use notation $\bL$
both for the class
$\hi_c[\bA^1]=[H_c^*(\bA^1,\bQ)]\in K(\MHM(\pt))$ \eqref{lef2}
and for the class
$[\bA^1]\in K(\Sch)$ \eqref{lef3}
depending on the context.
For a smooth (equidimensional) algebraic variety~$X$ of dimension $d$
(or a smooth Artin stack with affine stabilizers),
we define its virtual class
\begin{equation}\label{virt class}
[X]_\vir=\bL^{-d/2}[X]\in\K(\Sch).
\end{equation}
For example,
\begin{align}
[\bP^n]_\vir&=\bL^{-n/2}(1+\bL+\dots+\bL^n)
=\bL^{-n/2}+\bL^{-n/2+1}+\dots+\bL^{n/2},\\
[\GL_n]_\vir&=\bL^{-n^2/2}\prod_{i=0}^{n-1}(\bL^n-\bL^i)=\bL^{-n/2}\prod_{i=1}^n(\bL^i-1).
\end{align}
Note that $\ICV_X=\bQ_X\ang d=\bL^{-d/2}\bQ_X$, hence the class of
\[
H_c(X,\ICV_X)
=\bL^{-d/2}H_c(X,\bQ)\]
is equal to $\hi_c[X]_\vir$ which we will also denote by $[X]_\vir$ by abuse of notation.
\medskip

Let $\La$ be a commutative monoid admitting a filtration satisfying \eqref{filt} and $\La_1=\Lap=\La\ms\set0$.
Let $S=\bigsqcup_{\ga\in\La}S_\ga$ be a \La-graded commutative monoid in the category of complex algebraic varieties (so that $S_\ga$ are of finite type) such that the product $\mu:S_\ga\xx S_{\ga'}\to S_{\ga+\ga'}$ is finite.
We define
\[K(\MHM(S))=\bop\nolimits_{\ga\in\La}K(\MHM(S_\ga))\]
and equip it with the $\bZ$-graded (and $\La$-graded) \la-ring structure using the same formulas as before.
We extend this \la-ring structure to
$\K(\MHM(S))=\bop\nolimits_{\ga\in\La}\K(\MHM(S_\ga))$
using \eqref{ext la-str}.
There is a commutative diagram
\begin{equation}\label{CD-la-rings1}
\begin{tikzcd}
\K(\Sch\qt S)\rar["\deg_!"]\dar["\hi_c"']&
\K(\Sch\qt \bLa)=\K(\Sch)[\La]\dar["\hi_c"]\\
\K(\MHM(S))\rar["\deg_!"]
& \K(\MHM(\bLa))
=\K(\MHM(\pt))[\La]
\end{tikzcd}
\end{equation}
where $\hi_c$ is defined in \eqref{chi2},
$\deg:S\to\bLa$ is the degree map \eqref{deg map},
and all arrows are homomorphisms of $\La$-graded pre-\la-rings.
These maps induce homomorphisms between completions~\eqref{compl}.

As in \eqref{Exp}, we have the plethystic exponential
\begin{equation}\label{Exp2}
\Exp:\what\K_+(\MHM(S))\to 1+\what\K_+(\MHM(S)),\qquad
a\mto \sum_{n\ge0}\si^n(a).
\end{equation}
In terms of the Adams operations, the $\Exp$ and its inverse (called plethystic logarithm) are
\begin{equation}\label{Exp-Log}
\Exp(a)=\exp\rbigg{\sum_{n\ge1}\frac1n\psi^n(a)},\qquad
\Log(1+a)=\sum_{n\ge1}\frac{\mu(n)}n\psi^n(\log (1+a)),
\end{equation}
where $\mu$ is the M\"obius function.

\section{Moduli spaces of framed objects}
\label{sec:framed}
Let $\M_\ga$ (resp.\ $\Ms_\ga$) denote the moduli space of semistable
(resp.\ stable) vector bundles over a curve $C$ of type $\ga=(r,d)\in\bZ^2$.
The moduli space $\Ms_\ga$ is smooth,
while the moduli space~$\M_\ga$ can have singularities if $r$ and $d$ are not coprime.
On the other hand, we can construct the moduli space $\Mf_\ga$ of stable framed vector bundles,
consisting of pairs $(E,s)$, where $E$ is a vector bundle over $C$ of type $\ga$ and $s\in\Ga(C,E)$ is a section.
Under certain conditions, this moduli space is smooth and
the canonical morphism $\pi:\Mf_\ga\to \M_\ga$ is projective.
We will show that the  fibres of $\pi$ can be identified with moduli spaces of nilpotent quiver representations and we will use this result to prove that the morphism $\pi$ is virtually small
(see Definition~\ref{def:v-small}).
The properties of virtually small maps (see Theorem \ref{th:LT})
will be used later to analyse the degrees of the components of $\pi_*\ICV_{\Mf_\ga}$.
In what follows, we will introduce moduli spaces of framed objects in a more general setting than the case of vector bundles on a curve.

\subsection{Stability of framed objects}
Let $\Vecf$ be the category of finite-dimensional vector spaces over a field $K$.
Let $\cA$ be an abelian $K$-linear category and $\Phi:\cA\to\Vecf$ be a left exact functor, to be called a \idef{framing functor}.

\begin{definition}\label{def:fr}
A \idef{framed object} is a pair $(E,s)$, where $E\in\cA$ and $s\in\Phi(E)$.
We call it \idef{stable} if for every $F\sbe E$ with $s\in\Phi(F)$ we have $F=E$.
\end{definition}

Let $Z=-\sd+\sr\bi:\Ga=K(\cA)\to\bC$ be a stability function on an abelian category $\cA$ (see~\eg~\cite{bridgeland_stability}),
where $\sd,\sr\in\Hom_\bZ(\Ga,\bR)$.
For every $0\ne E\in\cA$, we have either $\sr(E)>0$ or $\sr(E)=0$ and $\sd(E)>0$.
We define the slope of $E$
$$\mu_Z(E)=\frac{\sd(E)}{\sr(E)}\in\bR\sqcup\set{+\infty}$$
and we say that an object $E$ is $Z$-semistable (resp.\ $Z$-stable) if, for any proper $0\ne F\sbs E$, we have $\mu_Z(F)\le\mu_Z(E)$ (resp.\ $\mu_Z(F)<\mu_Z(E)$).
For $\mu\in\bR$, let $\cA^\mu\sbe\cA$ be the subcategory of $Z$-semistable objects having slope $\mu$ (including the zero object).
It is a weakly Serre subcategory of~$\cA$, meaning an abelian subcategory closed under extensions, kernels and cokernels.

It will be convenient to extend the notion of framed objects
$(E,s)$ as follows.
Let $\cA_\f$ be the category consisting of triples $(E,V,s)$, where $E\in\cA$, $V\in\Vecf$ and $s:V\to \Phi(E)$ is linear.
The category $\cA_\f$ is an abelian category,
called the category of framed objects.
A pair $(E,s)$ with $s\in\Phi(E)$ can be identified with the triple $(E,K,s)\in\cA_\f$.
For $\eps>0$, we define the stability function on $\cA_\f$
\[Z_\eps(E,V,s)=Z(E)-\eps\cdot \dim V\]
and the corresponding notion of stability in $\cA_\f$.
A framed object $(E,s)$
is $Z_\eps$-stable for $0<\eps\ll \n{Z(E)}$
(we will say that it is $Z_+$-stable)
if and only if
\begin{enumerate}
\item $E$ is $Z$-semistable.
\item For every proper $F\sbs E$ with $s\in \Phi(F)$, we have $\mu_Z(F)<\mu_Z(E)$.
\end{enumerate}

\begin{lemma}\label{stab in C}
For a framed object $(E,s)$ with $E\in\cC=\cA^\mu$
\tfae
\begin{enumerate}
\item $(E,s)$ is $Z_+$-stable in $\cA_\f$.
\item $(E,s)$ is $Z_+$-stable in $\cC_\f$.
\item For every $F\sbe E$ in $\cC$ with $s\in \Phi(F)$, we have $F=E$.
\end{enumerate}
\end{lemma}
\begin{proof}
If $(E,s)$ is $Z_+$-stable in $\cA_\f$, then it is automatically $Z_+$-stable in $\cC_\f$.
If $(E,s)$ is $Z_+$-stable in $\cC_\f$ and $s\in\Phi(F)$ for a proper $F\sbs E$ in $\cC$, then $\mu_Z(F)<\mu_Z(E)$.
This contradicts to $F\in\cC=\cA^\mu$.
Assume that $(E,s)$ satisfies the third condition.
Let $s\in\Phi(F)$ for a proper $F\sbs E$ in $\cA$.
As $E\in\cA^\mu$, we have $\mu_Z(F)\le\mu$.
If $\mu_Z(F)=\mu$, then $F\in\cA^\mu=\cC$ and by our assumption $F=E$.
Otherwise $\mu_Z(F)<\mu_Z(E)$.
This implies that $(E,s)$ is $Z_+$-stable in $\cA_\f$.
\end{proof}

The above result implies that the notion of $Z_+$-stability of framed objects in $\cA^\mu$ is equivalent to the notion of stability of framed objects from Definition \ref{def:fr}.
To define such a framed object, it is enough to have
a left exact functor $\Phi:\cA^\mu\to\Vecf$.
Usually, this functor will be exact.

\begin{example}
\label{ex:q-rep}
Let $Q$ be a quiver and $\cA=\Rep Q$ be the category of $Q$-representations over a field~$K$.
For a fixed (framing) vector $\bw\in\bN^{Q_0}$, we define
the exact functor
\begin{equation}
\Phi:\cA\to\Vecf,\qquad M\mto\bop_{i\in Q_0}\Hom(K^{\bw_i},M_i)
\iso\bop_{i\in Q_0}M_i^{\oplus \bw_i}.
\end{equation}
The category $\cA_\f$ of framed objects can be identified with the category of representations of the new quiver $Q'$ obtained from $Q$ by adding a new vertex $*$ and $\bw_i$ arrows $*\to i$ for all $i\in Q_0$.
An object $(M,V,s)\in\cA_\f$ is identified with the representation $M'\in\Rep Q'$ such that $M'_i=M_i$ for $i\in Q_0$ and $M'_*=V$.
The linear map $s:V\to\Phi(M)=\bop_{i\in Q_0}M_i^{\oplus \bw_i}$ induces linear maps $M'_*\to M'_i$ corresponding to the arrows $*\to i$ in $Q'$.
A framed object $(M,s)$ corresponds to $M'\in\Rep Q'$ with $M'_*=K$.
It is stable if and only if $M'_*$ generates the whole representation.
\end{example}

\begin{example}\label{ex:VB}
Let $C$ be a smooth projective curve of genus $g$
and $\cA=\Coh C$ be the category of coherent sheaves over $C$.
We consider the stability function $Z(E)=-\deg E+\rk E\cdot \bi$
on~$\cA$.
For $\mu\in\bQ$, let $\cA^\mu=\Coh^\mu(C)\sbs\Coh(C)$ be the category of semistable vector bundles of slope~$\mu$.
For a fixed line bundle $L$ (or any coherent sheaf) over $C$,
we consider the left exact functor
\begin{equation}
\Phi:\cA\to\Vecf,\qquad E\mto \Hom(L,E).
\end{equation}
If the line bundle~$L$ has degree $\ell<\mu-(2g-2)$, then the functor
$\Phi:\cA^\mu\to\Vecf$ is exact.
Indeed, $\Ext^1(L,E)\iso \Hom(E,L\ts\om_X)^*=0$ for $E\in\cA^\mu$ since
$\mu>\ell+2g-2$.

On the other hand, for a fixed point $p\in C$,
the fibre functor
\begin{equation}\label{parab}
\Phi:\cA^\mu\to\Vecf,\qquad E\mto E(p)=E_p\ts_{\cO_{C,p}}K,
\end{equation}
is also exact.
Framed objects in this situation can be identified with parabolic vector bundles.
\end{example}

As before, let $\cA^\mu\sbe\cA$ be the category of $Z$-semistable objects of \cA having slope $\mu\in\bR$.
The stable objects of $\cA$ having slope $\mu$ are exactly the simple objects of $\cA^\mu$.
Let
\[E_1,\dots,E_n\in\cA^\mu\]
be a collection of (pairwise non-isomorphic) simple objects
and let $\cC=\ang{E_1,\dots,E_n}\sbs\cA^\mu$
be the abelian category generated by them.
It is a Serre subcategory of $\cA^\mu$, meaning an abelian subcategory closed under taking extensions, subobjects and quotients.
Given a left exact functor $\Phi:\cA^\mu\to\Vecf$,
we define the category $\cC_\f$ of framed objects in $\cC$ in the same way as before.

We are going to describe the categories $\cC$ and $\cC_\f$ as categories of quiver representations.
Let~$Q$ be the quiver with vertices $1,\dots,n$ and the number of arrows from $i$ to $j$ equal
\begin{equation}\label{Q1}
a_{ij}=\dim\Ext^1(E_i,E_j).
\end{equation}
Let $Q'$ be the quiver obtained from $Q$ by adding a new vertex $*$ and $\bw_i=\dim \Phi(E_i)$ arrows $*\to i$ for all $i\in Q_0$.
A representation of $Q$ is called nilpotent if it has a filtration by sub-representations such that the arrows of $Q$ act trivially on the factors.

\begin{theorem}
\label{rep nil}
\mbox{}
\begin{enumerate}
\item If $\cA^\mu$ is hereditary,
then the category $\cC$ is equivalent to the category $\Rep^\nil(Q)$ of nilpotent representations of $Q$.
\item If $\cA^\mu$ is hereditary and $\Phi:\cA^\mu\to\Vecf$ is exact, then the category of framed objects $\cC_\f$ is equivalent to $\Rep^\nil(Q')$.
\end{enumerate}
\end{theorem}
\begin{proof}
For the first statement see \cite[\S1.5]{deng_quiver}.
For the second statement we can assume that $\cC=\Rep^\nil(Q)$ and we will show that an exact functor $\Phi:\cC\to\Vecf$ is uniquely determined (up to a non-unique natural transformation) by its values on the simple objects.
This will imply the second statement as $\Rep^\nil(Q')$ can be identified with the category of framed objects in $\Rep^\nil(Q)$,
see Example \ref{ex:q-rep}.

Let $A=\bk Q$ be the path algebra,
$\fr\sbs A$ be the radical,
and $e_i\in A$ be the idempotent corresponding to the trivial path at a vertex $i\in Q_0$.
It is enough to show that $\Phi$ is uniquely determined on the subcategory $\cC_t=\Rep A_t\sbs\cC$ (which is not an exact subcategory), where $A_t=A/\fr^{t+1}$ for $t\ge0$.
Let $S_i$ be the simple modules and $P_i=A_te_i$ their projective covers as $A_t$-modules.
For the projective module
$P=\bop_i \Phi(S_i)^*\ts_\bk P_i$,
we consider the functor
$h_P=\Hom(P,-):\cC\to\Vecf$.
We will construct an isomorphism $a:h_P\to \Phi$ of functors over $\cC_t$.
The map
$$\Fun(h_P,\Phi)\iso \Phi(P)
=\bop_i \Phi(S_i)^*\ts_\bk \Phi(P_i)
\to\bop_i \Phi(S_i)^*\ts_\bk \Phi(S_i)
$$
is surjective as $P_i\to S_i$ is surjective and $\Phi$ is exact.
Consider a preimage $a:h_P\to \Phi$ of the sum of identity maps.
It corresponds to $a_i:\Phi(S_i)\to \Phi(P_i)$ such that the composition $\Phi(S_i)\to \Phi(P_i)\to \Phi(S_i)$ is the identity for all $i$.
It induces a chain of maps
\begin{multline*}
h_P(M)=\Hom(P,M)=\bop_i \Phi(S_i)\ts_\bk \Hom(P_i,M)\\
\to\bop_i \Phi(P_i)\ts_\bk \Hom(P_i,M)\to \Phi(M).
\end{multline*}
For $M=S_j$, the composition of the last two arrows is an isomorphism. This implies that $a:h_P\to \Phi$ induces an isomorphism on simple objects.
By exactness of $h_P$ and $\Phi$ on $\cC_t$, we conclude that $a:h_P\to \Phi$ is an isomorphism of functors on the whole category $\cC_t$.
\end{proof}

The above theorem implies that the moduli space of stable pairs $(E,s)$ with $E\in\cC$ can be identified with the moduli spaces of stable nilpotent representations of $Q'$ having dimension one at the vertex $*$ (see Example \ref{ex:q-rep}).
Stability of a representation $M\in\Rep Q'$ means
that $M_*$ generates the whole representation.

\begin{remark}
The second statement of the theorem implies that $\cC_\f$ is hereditary.
More generally, one can show that if $\cA$ is hereditary and $\Phi:\cA\to\Vecf$ is exact, then $\cA_\f$ is hereditary
\cite{mozgovoy_quiver}.
\end{remark}

\begin{example}
Let $S$ be a smooth projective surface and $H$ be an ample divisor such that $H\cdot K_S<0$, where $K_S$ is the canonical divisor of $S$.
Define the slope function
$$\mu(E)=\frac{H\cdot c_1(E)}{\rk E},\qquad E\in\Coh S.$$
For $\mu\in\bR$, let $\cA^\mu=\Coh^\mu S\sbs\Coh S$ be the category of semistable vector bundles (also called Mumford semistable) having slope $\mu$.
This category is hereditary.
Indeed, for any $E,F\in\cA^\mu$, we have
$\Ext^2(E,F)\iso\Hom(F,E\ts\om_S)^*=0$
as both $F$ and $E\ts\om_S$ are semistable and
$$\mu(E\ts\om_S)=\frac{H\cdot c_1(E)+\rk E\cdot H\cdot K_S}{\rk E}=\mu(E)+H\cdot K_S<\mu(F).$$
For $n\in\bZ$, we consider the left exact functor
\begin{equation}
\Phi_n:\cA^\mu\to\Vecf,\qquad E\mto \Hom(\cO(-nH),E).
\end{equation}
This functor is not exact in general,
but one can prove the following (\cf \cite[\S1.7, \S3.3]{huybrechts_geometry}).
For $E\in\Coh S$, let $\cl (E)=(\rk E,H\cdot c_1(E),c_2(E))\in\bZ^3$.
For any finite subset $C\sbs\bZ^3$, there exists $n>0$ such that
$R^i\Phi_n(E)=0$ for all $i>0$ and $E\in\cA^\mu$ with $\cl(E)\in C$.
\end{example}

\subsection{Moduli spaces}
\label{sec:MS1}
Let $\cA$ be an abelian ($\bC$-linear) category and $\cl:K(\cA)\to\Ga$ be a homomorphism to an abelian group $\Ga$ equipped with a bilinear form $\hi$ compatible with the Euler form on $\cA$
\begin{equation}
\hi(\cl E,\cl E')=\hi(E,E')=\sum_{i\ge0}(-1)^i\dim\Ext^i(E,E').
\end{equation}
Let $Z=-\sd+\sr\bi:\Ga\to\bC$ be a homomorphism inducing a stability function $K(\cA)\xto{\cl}\Ga\xto Z\bC$ (also denoted by $Z$).
For $\mu\in\bR$, let $\cA^\mu\sbe\cA$ be the category of $Z$-semistable objects having slope~$\mu$ (including the zero object).
We define the semigroups
\begin{equation}\label{La1}
\La=\sets{\cl(E)\in\Ga}{E\in\cA^\mu},\qquad
\Lap=\La\ms\set{0}.
\end{equation}
We assume that $\cA^\mu$ is hereditary and we have an exact functor
$\Phi:\cA^\mu\to\Vecf$
and a homomorphism $\vi:\La\to\bZ$ such that
\begin{equation}
\dim\Phi(E)=\vi(\cl E)\qquad\forall E\in\cA^\mu.
\end{equation}

\begin{example}
\label{ex:VB2}
As in Example \ref{ex:VB}, let $C$ be a smooth projective curve
of genus~$g$ over $\bC$
and let $\cA=\Coh C$ and $\cA^\mu=\Coh^\mu C\sbs\Coh C$
for a fixed $\mu\in\bQ$.
We consider the Chern character map $\ch:K(\cA)\to\Ga=\bZ^2$, $[E]\mto(\rk E,\deg E)$.
Let $\Phi=\Hom(L,-):\cA^\mu\to\Vecf$ for a line bundle $L$ of degree $\ell<\mu-(2g-2)$.
For $\ga=\ch(E)=(r,d)$ and $\ga'=\ch E'=(r',d')$, we have
\begin{gather}
\hi(\ga,\ga')=\hi(E,E')=(1-g)rr'+(rd'-r'd),\label{euler1}\\
\vi(\ga)=\hi(L,E)=(1-g-\ell)r+d.
\end{gather}
Note that the Euler form is symmetric on $\La$.
Indeed, if $\ga,\ga'\in\La$, then $rd'=r'd$, hence
$\hi(\ga,\ga')=(1-g)rr'=\hi(\ga',\ga)$.
\end{example}

We assume that for every $\ga\in\La$ there exists a moduli space $\M_\ga$ of semistable objects in $\cA$ having class $\ga$
and a moduli space $\Mf_\ga$ of stable framed objects $(E,s)$ with $E$ having class $\ga$
(see \eg \cite{huybrechts_framed} in the case of framed objects on a curve).
The canonical projection $\pi:\Mf_\ga\to\M_\ga$ is a projective map.
The moduli space $\Mf_\ga$ is smooth as all objects are stable and the second \Ext-group vanishes (see \eg \cite{thaddeus_stable}).
The moduli space $\M_\ga$ parametrizes poly-stable objects of the form
$$E=\bop_{i=1}^n E_i^{m_i},$$
where $E_i$ are pairwise non-isomorphic stable objects in $\cA$
(simple objects in $\cA^\mu$)
with $\ga_i=\cl E_i\in\Lap$ such that $\ga=\sum_i m_i\ga_i$.
We define the type of $E$ to be the pair $(\bga,\m)$ with
$$\bga=(\ga_i)_{i=1}^n\in(\Lap)^n,\qquad
\qquad  \m=(m_i)_{i=1}^n\in\bN^n,$$
The set $S_{\bga,\m}\sbs \M_\ga$ of all objects having type $(\bga,\m)$ is called the Luna stratum of type $(\bga,\m)$.
In our applications they form a finite partition of $\M_\ga$ into locally closed subsets.
If $n=1$ and $m_1=1$, then $S_{\bga,\m}$ is equal to the stable locus $\Ms_\ga\sbe\M_\ga$ which is smooth and open.
The fibre of $\pi$ over $E\in\Ms$ is isomorphic to the projective space
$P(\Phi(E))\iso\bP^{\vi(\ga)-1}$.

Following \eqref{Q1},
we define the quiver $Q_{\bga}$ with vertices $1,\dots,n$ and the number of arrows from $i$ to $j$ equal
\begin{equation}\label{fiber quiver}
a_{ij}=\dim\Ext^1(E_i,E_j)=\de_{ij}-\hi(\ga_i,\ga_j).
\end{equation}
We define $Q'_{\bga}$ to be the quiver obtained from $Q_{\bga}$ by adding one vertex $*$ and adding
\begin{equation}
\bw_i=\dim\Phi(E_i)=\vi(\ga_i)
\end{equation}
arrows $*\to i$ for all vertices $i$ in $Q_{\bga}$.

\begin{remark}\label{VB symmetric}
In the case of vector bundles on a curve and $\ga_i=\ch E_i=(r_i,d_i)$ with $d_i/r_i=\mu$, we obtain
$a_{ij}=\de_{ij}+(g-1)r_ir_j$.
Note that the corresponding quiver is symmetric ($a_{ij}=a_{ji}$).
\end{remark}
The Euler form of $Q_{\bga}$ is given by
$$\hi_{Q_{\bga}}(\m,\m')=\sum_{i}m_im'_i-\sum_{i,j}a_{ij}m_im'_j
=\sum_{i,j}(\de_{ij}-a_{ij})m_im'_j
,\qquad \m,\m'\in\bZ^n.$$
If $\ga=\sum_i m_i\ga_i$ and $\ga'=\sum_i m'_i\ga_i$, then
\[\hi(\ga,\ga')=\sum_{i,j}\hi(\ga_i,\ga_j)m_im'_j
=\hi_{Q_{\bga}}(\m,\m').\]
This means that the map $\bZ^n\to\Ga$, $\m\mto\sum_i m_i\ga_i$ preserves the Euler forms.

\begin{theorem}
\label{th:fiber}
For any object $E\in \M_\ga$ of type $(\bga,\m)$,
the fibre $\pi\inv(E)\sbs \Mf_\ga$ is isomorphic to the moduli space
$\M^{\f,\nil}_\m$ of stable nilpotent representations $M$ of $Q'_{\bga}$ with $M|_{Q_{\bga}}$ having dimension vector $\m$ and $\dim M_*=1$.
Stability of $M$ means that $M_*$ generates the whole representation.
\end{theorem}
\begin{proof}
Let $E=\bop_{i=1}^n E_i^{m_i}$, where $\cl E_i=\ga_i\in\Lap$.
Let $\cC\sbs\cA^\mu$ be the abelian category generated by $E_1,\dots,E_n$.
We proved in Theorem \ref{rep nil} that $\cC$ is equivalent to the category $\Rep^\nil(Q_{\bga})$ and $\cC_\f$ is equivalent to the category $\Rep^\nil(Q'_{\bga})$.

If $(E',s)\in\pi\inv(E)$, then the factors of the Jordan-H\"older filtration of $E'\in\cA^\mu$ contain $m_i$ copies of $E_i$ for all $1\le i\le n$.
Therefore $E'\in\cC$ and $(E',s)$ is stable in $\cC_\f$
(\cf Lemma \ref{stab in C}).
As $\cC_\f\iso\Rep^\nil(Q'_{\bga})$,
we can identify $(E',s)$ with a stable nilpotent representation $M$ of $Q'_{\bga}$ such that $M|_{Q_{\bga}}$ has dimension vector $\m$ and $\dim M_*=1$.
\end{proof}

\subsection{Virtual smallness}
Let us assume now that the Euler form is symmetric on $\La$.

\begin{theorem}
\label{framed-semism}
The defect of $\pi:\Mf_\ga\to \M_\ga$ over $S_{\bga,\m}$ is
$\le \vi(\ga)-1$,
with the equality only over the stable locus $\Ms_\ga\sbe \M_\ga$
(if $\Ms_\ga\ne\es$).
\end{theorem}
\begin{proof}
Let $S_{\bga,\m}\sbe \M_\ga$ be the stratum corresponding
to $\bga=(\ga_1,\dots,\ga_n)\in\Ga^n$
and $\m=(m_1,\dots,m_n)\in\bN^n$ such that $\sum_i m_i\ga_i=\ga$.
Let $Q=Q_{\bga}$ and $Q'=Q'_{\bga}$ be as in \eqref{fiber quiver}.
By our assumptions, the quiver $Q$ is symmetric (\cf Remark \ref{VB symmetric}).
By Theorem \ref{th:fiber},
for $E=\bop_{i=1}^nE_i^{m_i}\in S_{\bga,\m}$,
we can identify the fibre $\pi\inv(E)$
with the moduli space $\M^{\f,\nil}_\m$
of stable nilpotent representations of $Q'$.
Let $R_\m=\bop_{a:i\to j}\Hom(\bC^{m_i},\bC^{m_j})$
be the space of all representations of $Q$ having dimension vector $\m$ and let $R^\nil_\m\sbs R_\m$ be the subscheme of nilpotent representations.
These schemes are equipped with the action of the group $G_\m=\prod_{i\in Q_0}\GL_{m_i}$.
As $Q$ is a symmetric quiver, we have
(see \cite{meinhardt_donaldson})
$$\toh\hi_Q(\m,\m)+\dim R^\nil_\m-\dim G_\m\le \sum_{i=1}^nm_i\rbr{\toh\hi_Q(e_i,e_i)-1},$$
where $e_i$ is the standard basis vector of $\bZ^{n}$.
We have
$$\dim\M^{\f,\nil}_\m
=\dim R^\nil_\m(Q)-\dim G_\m+\bw\cdot \m
$$
as $G_\m$ acts freely on stable framed representations.
Note that $\bw\cdot\m=\vi(\ga)$.
Therefore
\[\dim \M_\m^{\f,\nil}\le
\sum_{i=1}^nm_i\rbr{\toh\hi(\ga_i,\ga_i)-1}
-\toh\hi(\ga,\ga)+\vi(\ga).
\]
The dimension of the stable locus $\Ms_\ga\sbe\M_\ga$ is $1-\hi(\ga,\ga)$, hence
$\dim S_{\bga,\m}=\sum_{i=1}^n(1-\hi(\ga_i,\ga_i))$ (if $S_{\bga,\m}$ is non-empty).
The defect of $\pi$ over $S_{\bga,\m}$ is equal (see \eqref{def1})
\begin{multline}\label{long ineq}
\de(\pi,S_{\bga,\m})
=2\dim \M_\m^{\f,\nil}+\dim S_{\bga,\m}-\dim\Mf_\ga\\
=2\dim \M_\m^{\f,\nil}
+\sum_{i=1}^n(1-\hi(\ga_i,\ga_i))
+\rbr{\hi(\ga,\ga)-\vi(\ga)}\\
\le\sum_{i=1}^nm_i\rbr{\hi(\ga_i,\ga_i)-2}
+\sum_{i=1}^n(1-\hi(\ga_i,\ga_i))+\vi(\ga)\\
=\sum_{i=1}^n(m_i-1)\rbr{\hi(\ga_i,\ga_i)-1}
-\sum_{i=1}^nm_i+\vi(\ga)
\le -\sum_{i=1}^nm_i+\vi(\ga)\le \vi(\ga)-1,
\end{multline}
where we used the fact that $m_i\ge1$ and $\hi(\ga_i,\ga_i)\le1$.
The last inequality in \eqref{long ineq} can be an equality only for $n=1$ and $m_1=1$.
In this case the stratum $S_{\bga,\m}$ is the stable locus $\Ms_\ga$.
The fibre over $E\in \Ms_\ga$ can be identified with the projective space $P(\Phi(E))\iso\bP^{\vi(\ga)-1}$.
Therefore the defect of $\pi$ over $\Ms_\ga$ is
\[\de(\pi,\Ms_\ga)
=2(\vi(\ga)-1)+(1-\hi(\ga,\ga))+(\hi(\ga,\ga)-\vi(\ga))
=\vi(\ga)-1.
\]
\end{proof}

\section{DT invariants}
\label{sec:main}

\subsection{Moduli spaces}
We follow conventions from \S\ref{sec:MS1}.
Let $C$ be a smooth projective curve of genus $g$ over $\bC$.
Let $\cA=\Coh C$ and $\cA^\mu=\Coh^\mu C\sbe \Coh C$ be the category of semistable vector bundles having a fixed slope $\mu\in\bQ$.
Similarly to \eqref{La1}, we define
\[\Lap=\sets{(r,d)\in\bZ_{>0}\xx\bZ}{d/r=\mu},\qquad \La=\Lap\sqcup\set{0}.\]
Note that $\La=\bN\ga_0\iso\bN$, for a unique element $\ga_0=(r_0,d_0)\in\Lap$.

For $\ga=(r,d)\in\La$, let $\M_\ga$ (resp.\ $\Ms_\ga$) be the moduli space of semistable (resp.\ stable) vector bundles over $C$ having class $\ga$ (rank $r$ and degree $d$).
Similarly, let $\fM_\ga$ be the moduli stack of semistable vector bundles over~$C$ having class $\ga$.

\begin{remark}
The moduli space $\M_\ga$ is always irreducible
\lep[Theorem 8.5.2].
If the stable locus $\Ms_\ga\sbe\M_\ga$ is nonempty,
then it is smooth of dimension $1-\hi(\ga,\ga)=(g-1)r^2+1$.
Moreover,
\begin{enumerate}
\item
If $g\ge 2$, then $\Ms_\ga\ne\es$ \lep[Theorem 8.6.1].
\item
If $g=1$, then $\M_\ga\iso S^k(C)$
for $k=\gcd(r,d)$ \lep[Theorem 8.6.2].
We have $\Ms_\ga=\M_\ga$ if $k=1$ and $\Ms_\ga=\es$ otherwise.
Indeed, if $\Ms_\ga\ne\es$, then $\dim S^k(C)=1$, hence $k=1$.

\item
If $g=0$, then $\M_\ga=\pt$ if $d/r\in\bZ$ and $\M_\ga=\es$ otherwise.
We have $\Ms_\ga=\M_\ga$ if $r=1$ and $\Ms_\ga=\es$ otherwise.
\end{enumerate}
\end{remark}

\begin{remark}
Note that the restriction of the Euler form \eqref{euler1} to $\La$ (or $\cA^\mu$) is symmetric.
All of the results below translate verbatim to moduli spaces in other hereditary categories under the assumption that the restriction of the Euler form to $\cA^\mu$ is symmetric.
\end{remark}

As in Example \ref{ex:VB2}, we fix a line bundle $L$ of degree $\ell<\mu-(2g-2)$ and consider the exact functor
$\Phi=\Hom(L,-):\cA^\mu\to\Vecf$
and the homomorphism $\vi:\La\to\bZ$
such that $\dim\Phi(E)=\vi(\ch E)$ for all $E\in\cA^\mu$.
Let $\Mf_\ga$ be the moduli space of stable framed objects $(E,s)$
with $E\in\cA^\mu$ having class $\ga$ and $s\in\Phi(E)$.
We define
\[\M=\bigsqcup_{\ga\in\La}\M_\ga,\qquad
\fM=\bigsqcup_{\ga\in\La}\fM_\ga,\qquad
\Mf=\bigsqcup_{\ga\in\La}\Mf_\ga,\]
and consider the commutative diagram
\[\begin{tikzcd}
\Mf\ar[rr,"q"]\ar[rd,"\pi"']&& \MM\ar[dl,"p"]\\
&\M
\end{tikzcd}\]

The moduli space $\M$ is a $\La$-graded commutative monoid in the category of algebraic schemes over $\bC$,
where the product map $\mu:\cM\xx\cM\to \cM$ is given by the direct sum of vector bundles and is a finite map.
The unit map is $\eta:\pt=\M_0\emb \M$.
By the results of \S\ref{sec:monoids}, we have a $\La$-graded pre-\la-ring $\K(\Sch\qt\M)$
and a $\La$-graded (and $\bZ$-graded) \la-ring $\K(\MHM(\M))$.
As in \eqref{CD-la-rings1}, we  have a commutative diagram
\[\begin{tikzcd}
\K(\Sch\qt\M)\rar["\deg_!"]\dar["\hi_c"']&\K(\Sch\qt\bLa)=\K(\Sch)[\La]\dar["\hi_c"]\\
\K(\MHM(\M))\rar["\deg_!"]&\K(\MHM(\bLa))=\K(\MHM(\pt))[\La]\rar["E"]
&\bQ(u^\oh,v^\oh)[\La]
\end{tikzcd}\]
where $\deg:\M\to\bLa$ is the degree map \eqref{deg map},
$E$ is the $E$-polynomial map (see \S\ref{sec:HD}),
and all arrows are homomorphisms of pre-\la-rings.

\subsection{Motivic DT invariants}
The (absolute) motivic DT invariants
$\aDTm_\ga\in \K(\Sch)$ are defined by the formula in $\what\K(\Sch\qt\bLa)\iso\prod_{\ga\in\La}\K(\Sch)t^\ga$
\[\sum_{\ga\in\La}\bL^{\oh\hi(\ga,\ga)}[\MM_\ga]t^\ga
=\Exp\rbigg{\frac{\sum_{\ga\in\Lap}\aDTm_\ga t^\ga}{\Gv}}.\]
Note that $\dim\MM_\ga=-\hi(\ga,\ga)$,
hence $\bL^{\oh\hi(\ga,\ga)}[\MM_\ga]=[\MM_\ga]_\vir$
\eqref{virt class}.
Note also that $\Gv=[\Gm]_\vir$.
Similarly, we define the (relative) motivic DT invariants
$\DTm_\ga\in \K(\Sch\qt\M_\ga)$
by the formula in $\what\K(\Sch\qt\M)$
\begin{equation}\label{mot DT}
\sum_{\ga\in\La}\bL^{\oh\hi(\ga,\ga)}[\MM_\ga\to\M_\ga]
=\Exp\rbigg{\frac{\sum_{\ga\in\Lap}\DTm_\ga}{\Gv}}.
\end{equation}
We define
$\DTm=\sum_{\ga\in\Lap}\DTm_\ga\in\what\K(\Sch\qt\M)$.

\begin{theorem}[Integrality conjecture \cite{meinhardt_donaldsona}]
\label{th:int}
The class $\DTm_\ga$ is contained in the image of the map
$K(\Sch\qt\M_\ga)[\bL^{-1/2}]\to \K(\Sch\qt\M_\ga).$
\end{theorem}

\subsection{Mixed DT invariants}
The (absolute) DT invariants are defined by
\begin{equation}
\aDT_\ga=\hi_c(\aDTm_\ga)\in \K(\MHS^p),
\end{equation}
where $\hi_c:\K(\Sch)\to\K(\MHS^p)$
was defined in \eqref{chi2}.
Similarly, the relative DT invariants are
\begin{equation}
\DT_\ga=\hi_c(\DTm_\ga)\in \K(\MHM(\M_\ga)).
\end{equation}
Using the fact that
$\hi_c:\K(\Sch\qt\bLa)\to\K(\MHM(\bLa))$
and $\hi_c:\K(\Sch\qt\M)\to\K(\MHM(\M))$ are morphisms of pre-\la-rings, we can write
\begin{gather}
\sum_{\ga\in\La}\bL^{\oh\hi(\ga,\ga)}\hi_c[\MM_\ga]
=\Exp\rbigg{\frac{\sum_{\ga\in\Lap}\aDT_\ga}{\Gv}}\label{aDT2}\\
\sum_{\ga\in\La}\bL^{\oh\hi(\ga,\ga)}\hi_c[\MM_\ga\to\M_\ga]
=\Exp\rbigg{\frac{\sum_{\ga\in\Lap}\DT_\ga}{\Gv}}.
\end{gather}

\begin{corollary}
\label{cor:int}
The class $\DT_\ga$ is contained in $K(\MHM(\M_\ga))[\bL^{1/2}]$.
\end{corollary}
\begin{proof}
By Theorem \ref{th:int}, the class $\DT_\ga$ is contained in the image of the map (note that $\bL$ is invertible)
$K(\MHM(\M_\ga))[\bL^{1/2}]\to \K(\MHM(\M_\ga))$.
This map is injective as $K(\MHM(\M_\ga))$ is free over $\bZ[\bL^{\pm1}]$ (the basis consists of isomorphism classes of simple objects in $\MHM(\cM_\ga)$ having weight $0$ or $1$).
\end{proof}

\begin{remark}
The stack $\fM_\ga$ can be represented as a global quotient $X/G$, where $X$ is smooth and $G=\GL_n$ is a general linear group
(see \eg \cite{huybrechts_geometry}).
If $q:X\to Y$ is a principal $G$-bundle between smooth algebraic varieties, then $[\ICV_Y]=[G]_\vir\inv \cdot q_![\ICV_X]$.
In our case, we consider the diagram
\[\begin{tikzcd}
X\ar[rr,"q"]\ar[rd]&& \MM_\ga\ar[dl,"p"]\\
&\M_\ga
\end{tikzcd}\]
and formally define
$$p_![\ICV_{\fM_\ga}]=[G]_\vir\inv\cdot(pq)_![\ICV_X]
\in \K(\MHM(M_\ga)).$$
Then
$p_![\ICV_{\fM_\ga}]
=\bL^{-d_X/2+d_G/2}[G]\inv\cdot (pq)_![\bQ_X]
=\bL^{\oh\hi(\ga,\ga)}\cdot \hi_c[\MM_\ga\to\M_\ga]$.
Therefore
\[\sum_{\ga\in\La}p_![\ICV_{\fM_\ga}]
=\Exp\rbigg{\frac{\sum_{\ga\in\Lap}\DT_\ga}{\Gv}}.\]
\end{remark}

\subsection{DT invariants and framed moduli spaces}
In this section we will express DT invariants using moduli spaces of stable framed objects.
This expression is a version of a wall-crossing formula, also called a DT/PT correspondence.
Recall that we have a commutative diagram
\[\begin{tikzcd}
\Mf\ar[rr,"q"]\ar[rd,"\pi"']&& \MM\ar[dl,"p"]\\
&\M
\end{tikzcd}\]

\begin{theorem}
\label{th:dt def2}
We have
\begin{equation}\label{DT2}
\pi_!\rbigg{\sum_{\ga\in\La}(-1)^{\vi(\ga)}[\ICV_{\Mf_\ga}]}
=\Exp\rbigg{\sum_{\ga\in\Lap}(-1)^{\vi(\ga)}[\bP^{\vi(\ga)-1}]_\vir\DT_\ga}.
\end{equation}
\end{theorem}

\begin{proof}
We will prove a similar formula in $\what\K(\Sch\qt\M)$
and then apply the map $\hi_c:\what\K(\Sch\qt\M)\to \what\K(\MHM(\M))$.
The wall-crossing formula for framed objects
(\cf \cite{engel_smooth,mozgovoy_wall-crossing,bridgeland_hall}) implies
\[\sum_{\ga}\bL^{\oh\hi(\ga,\ga)+\vi(\ga)}[\fM_\ga\to\M_\ga]
=
{\sum_{\ga}\bL^{\oh\hi(\ga,\ga)}}[\Mf_\ga\to\M_\ga]\cdot
\sum_{\ga}\bL^{\oh\hi(\ga,\ga)}[\fM_\ga\to\M_\ga].
\]
Applying \eqref{mot DT}, we obtain
\[
\Exp\rbr{\frac{\sum_{\ga}\bL^{\vi(\ga)}\DTm_\ga}{\Gv}}
={\sum_\ga \bL^{\oh\hi(\ga,\ga)}[\Mf_\ga\to\M_\ga]}\cdot
\Exp\rbr{\frac{\sum_{\ga}\DTm_\ga\cdot}{\Gv}}.
\]
Let $d_\ga=\dim \Mf_\ga=-\hi(\ga,\ga)+\vi(\ga)$
and $z_\ga=\bL^{-d_\ga/2}[\Mf_\ga\to\M_\ga]$.
Then
\[\sum_\ga \bL^{\oh\vi(\ga)}z_\ga
=\Exp\rbigg{\frac{\sum_{\ga}(\bL^{\vi(\ga)}-1)\DTm_\ga}{\Gv}}
=\Exp\rbigg{\sum_{\ga}\bL^{\oh}[\bP^{\vi(\ga)-1}]\DTm_\ga}.
\]
We apply the pre-\la-ring morphism
$\sum_\ga x_\ga\mto\sum_\ga (-\bL^{1/2})^{-\vi(\ga)}x_\ga$
to both sides and obtain
\[\sum_\ga (-1)^{\vi(\ga)}z_\ga
=\Exp\rbigg{\sum_{\ga}(-1)^{\vi(\ga)}[\bP^{\vi(\ga)-1}]_\vir\DTm_\ga}.
\]
Finally, we apply $\hi_c$ to both sides
and note that $\hi_c(z_\ga)=[\pi_!(\bQ_{\Mf_\ga}\ang {d_\ga})]
=\pi_![\ICV_{\Mf_\ga}]$.
\end{proof}

\subsection{DT invariants and intersection complexes}
For $\ga=(r,d)\in\La$, the stable locus $\Ms_\ga\sbe\M_\ga$ is open, smooth and has dimension
$1-\hi(\ga,\ga)$ (if $\Ms_\ga\ne\es$).
We consider the object $\ICV_{\ubar{\Ms_\ga}}\in\MHM(\M_\ga)$ which is defined to be zero if $\Ms_\ga=\es$.

\begin{theorem}
\label{main:proof}
We have $\DT_\ga=[\ICV_{\ubar{\Ms_\ga}}]$ for $\ga\in\La$.
\end{theorem}
\begin{proof}
The formula \eqref{DT2} can be written in the form
\begin{equation}\label{DT3}
(-1)^{\vi(\ga)}\pi_*[\ICV_{\Mf_\ga}]
=\sum_{\ov{m:\Lap\to\bN}{\sum_\al m_\al\al=\ga}}
\prod_{\al\in\Lap} \si^{m_\al}\rbr{(-1)^{\vi(\al)}[\bP^{\vi(\al)-1}]_\vir\DT_{\al}}.
\end{equation}
We will compare the highest degree terms on both sides.
By Theorem \ref{framed-semism},
the map $\pi:\Mf_\ga\to \M_\ga$ is $n$-small, where $n=\vi(\ga)-1$.
It is a $\bP^n$-fibration over $\Ms_\ga$.
By Theorem \ref{th:LT}, the object $\pi_*\ICV_{\Mf_\ga}$ has degree $\le n$ and the degree $n$ component
$\ICV_{\ubar {\Ms_\ga}}\ang{-n}
=\bL^{n/2}\ICV_{\ubar {\Ms_\ga}}$.

Let us consider the summand on the right corresponding to $m:\Lap\to\bN$ different from the delta-function $\de_\ga:\Lap\to\bN$.
By induction on $\ga$, the DT invariants in this summand
satisfy $\DT_\al=[\ICV_{\ubar{\Ms_\al}}]$, hence have degree zero (if $\DT_\al\ne0$).
Therefore
$z_\al=(-1)^{\vi(\al)}[\bP^{\vi(\al)-1}]_\vir\DT_{\al}$ has degree
$\le \vi(\al)-1$.
As $\K(\MHM(\M))$ is a weight graded  \la-ring
(see Theorem \ref{th:la-MHM}),
we conclude that $\si^{m_\al}(z_\al)$ has degree $\le m_\al(\vi(\al)-1)$.
Therefore the summand $\prod_\al \si^{m_\al}(z_\al)$
has degree $\le\sum_\al m_\al(\vi(\al)-1)
=\vi(\ga)-\sum_\al m_\al<\vi(\ga)-1=n$.

The summand on the right corresponding to the delta-function $\de_\ga:\Lap\to\bN$ is equal to $z_\ga=(-1)^{\vi(\ga)}[\bP^{\vi(\ga)-1}]_\vir\DT_\ga$.
By Corollary \ref{cor:int}, we have $\DT_\ga\in K(\MHM(\M_\ga))[\bL^{\oh}]$,
hence we can write $\DT_\ga=\sum_{i\in\bZ}x_i$, where $x_i$ is homogeneous of degree $i$.
We can prove by induction that the class $\DT_\ga$ is self-dual, hence $\DT_\ga$ has degree $r\ge0$ (if $\DT_\ga\ne 0$).
Then $z_\ga$ has degree $r+n$ and the leading term $(-1)^{\vi(\ga)}\bL^{n/2}x_r$ (if $\DT_\ga\ne 0$).

We conclude that if $\DT_\ga=0$, then the right hand side has degree $<n$, hence $\ICV_{\ubar {\Ms_\ga}}=0$ and $\Ms_\ga=\es$.
If $\DT_\ga\ne0$, then the right hand side has degree $r+n$ and the leading term $(-1)^{\vi(\ga)}\bL^{n/2}x_r$.
Comparing it to the left hand side, we conclude that $r=0$ and
$x_0=[\ICV_{\ub{\Ms_\ga}}]$.
As $\DT_\ga$ is self-dual and $x_i=0$ for $i>0$,
we conclude that
$\DT_\ga=x_0=[\ICV_{\ub{\Ms_\ga}}]$.
\end{proof}

The above result translates verbatim to moduli spaces in other hereditary categories under the assumption that the restriction of the Euler form to $\cA^\mu$ is symmetric
(so that we can apply Theorem~\ref{framed-semism}).

\begin{corollary}[\cf Theorem \ref{main2}]
\label{cor-main2}
We have
\begin{gather}
\DT_\ga
=\begin{cases}
[\ICV_{\M_\ga}]& \Ms_\ga\ne\es,\\
0&\Ms_\ga=\es.
\end{cases}\label{DT-IC}\\
\aDT_\ga
=\begin{cases}
[H^*(\M_\ga,\ICV_{\M_\ga})]& \Ms_\ga\ne\es,\\
0&\Ms_\ga=\es.
\end{cases}\label{aDT-IC}
\end{gather}
\end{corollary}
\begin{proof}
The moduli space $\M_\ga$ is irreducible and $\Ms_\ga\sbe\M_\ga$ is open.
If $\Ms_\ga\ne\es$, then
$\ub{\Ms_\ga}=\M_\ga$, hence
$\DT_\ga=[\ICV_{\ub{\Ms_\ga}}]=[\ICV_{\M_\ga}]$
and we obtain \eqref{DT-IC}.
To prove \eqref{aDT-IC},
we apply $a_!:\K(\MHM(\M_\ga))\to\K(\MHM(\pt))$ for the projection $a:\M_\ga\to\pt$.
\end{proof}

\begin{example}
If $C$ is an elliptic curve, then $\M_\ga\iso S^k(C)$ for $k=\gcd(r,d)$, $\ga=(r,d)$.
Moreover, $\Ms_\ga=\M_\ga$ if $k=1$ and $\Ms_\ga=\es$ otherwise.
Therefore
\[\DT_{\ga}
=\begin{cases}
[\ICV_{\M_\ga}]=\bL^{-1/2}[\bQ_{\M_\ga}]&\gcd(r,d)=1,\\
0&\text{otherwise}.
\end{cases}
\]
\end{example}

\subsection{Hodge-Deligne polynomials}
\label{sec:HD}
Given a mixed Hodge structure $V$, we define its Hodge-Deligne polynomial (also called Hodge-Euler polynomial or $E$-polynomial)
\begin{equation}
E(V;u,v)=\sum_{p,q}h^{p,q}(V)u^pv^q,\qquad
h^{p,q}(V)=\dim\Gr_F^p\Gr^W_{p+q}(V_\bC).
\end{equation}
It induces a \la-ring homomorphism
\begin{equation}
E:K(\MHM(\pt))=K(\MHS^p(\bQ))\to\bZ[u^{\pm1},v^{\pm1}],
\end{equation}
where the \la-ring structure on the right is given by Adams operations
$\psi^n(f(u,v))=f(u^n,v^n)$.
The fact that the above map preserves \la-ring structures follows from the fact that it is induced by an exact monoidal functor
$\MHS^p(\bQ)\to\Vecf_\bC^{\bZ^2}$, $V\mto V_\bC$
(where $\Vecf_\bC^{\bZ^2}$ denotes the category of $\bZ^2$-graded objects in $\Vecf_\bC$) and that
$K(\Vecf_\bC^{\bZ^2})\iso\bZ[u^{\pm1},v^{\pm1}]$.
For an algebraic variety $X$,
we have the class
$\hi_c[X]=[H_c^*(X,\bQ)]=\sum_i(-1)^i[H^i_c(X,\bQ)]\in K(\MHM(\pt))$ and we define
\begin{equation}
E(X)=E(\hi_c[X])
=\sum_i(-1)^i E(H^i_c(X,\bQ))
\in\bZ[u^{\pm1},v^{\pm1}].
\end{equation}
In particular, $\bL=H^*_c(\bA^1,\bQ)=\bQ(-1)[-2]$ and
$E(\bL)=E(\bQ(-1))=uv$.
We extend $E$ to a \la-ring homomorphism
\[E:\K(\MHM(\pt))\to \bQ(u^{1/2},v^{1/2})\]
with $E(\bL^{1/2})=-(uv)^{1/2}$
(the minus sign corresponds to the fact that
$\bL^{1/2}=\bQ(-1/2)[-1]$ has odd homological degree).
In what follows we will denote $E(\bL^{1/2})$ by $\bL^{1/2}$.
For a finite type Artin stack $X$ over $\bC$ with affine stabilizers, we have the classes
$[X]\in\K(\Sch\qt\pt)$, $\hi_c[X]\in\K(\MHM(\pt))$
and we define $E(X)=E(\hi_c[X])$.

For $\mu\in\bQ$, we define the series (\cf \eqref{ser Q1})
\begin{equation}
Q_\mu=1+\sum_{d/r=\mu}Q_{r,d}t^r
=1+\sum_{d/r=\mu}\bL^{(1-g)r^2/2}E(\MM_{r,d})t^r
\in\bQ(u^\ohh,v^\ohh)\pser t
\end{equation}
Note that $\dim\MM_\ga=-\hi(\ga,\ga)=(g-1)r^2$ for $\ga=(r,d)$.
As in the introduction, we define the DT invariants of the curve by the formula \eqref{int:DT}
\begin{equation}\label{Q2}
Q_\mu=\Exp\rbigg{\frac{\sum_{d/r=\mu}\HDT_{r,d}t^r}{\Gv}}.
\end{equation}

\begin{theorem}[\cf Theorem \ref{main1}]
\label{main1-proof}
If $\Ms_\ga=\es$, then $\HDT_\ga=0$.
If $\Ms_\ga\ne\es$, then
\begin{gather}
\HDT_\ga=E(H^*(\M_\ga,\ICV_{\M_\ga}))
=\bL^{-\dim \M_\ga/2}E(\IH^*(\M_\ga,\bQ))\\
\HDT_{\ga}(y,y)
=(-y)^{-\dim \M_\ga}\sum\nolimits_k\dim\IH^k(\M_\ga,\bQ)(-y)^k,
\end{gather}
where $\dim\M_\ga=\dim\Ms_\ga=(g-1)r^2+1$ for $\ga=(r,d)\in\bZ^2$.
\end{theorem}
\begin{proof}
Comparing \eqref{Q2} and \eqref{aDT2}, we conclude that
$\HDT_\ga=E(\aDT_\ga)$.
Taking the $E$-polynomials on both sides of \eqref{aDT-IC}, we obtain
$\HDT_\ga
=\begin{cases}
E(H^*(\M_\ga,\ICV_{\M_\ga}))
& \Ms_\ga\ne\es,\\
0&\Ms_\ga=\es.
\end{cases}$

We have (\cf \eqref{i-coh2})
\[\IH^*(\M_\ga,\bQ)=H^*(\M_\ga,\ICV_{\M_\ga}\ang{-\dim \M_\ga})
=\bL^{\dim{\M_\ga}/2}H^*(\M_\ga,\ICV_{\M_\ga}),\]
hence $E(H^*(\M_\ga,\ICV_{\M_\ga}))=\bL^{-\dim\M_\ga/2}E(\IH^*(\M_\ga,\bQ))$.
The second formula follows by substitution $u=v=y$, $\bL^{1/2}=-y$ and the fact
that $\IH^k(\M_\ga,\bQ)$ is pure of weight $k$ (note that $\ang{1}$ doesn't change the weight).
\end{proof}

If $g\ge2$, then $\Ms_\ga\ne\es$, hence
$\dim \M_\ga=(g-1)r^2+1$ for $\ga=(r,d)$.
Therefore
\begin{equation}
\sum_{d/r=\mu}\bL^{(1-g)r^2/2}
E(\IH^*(\M_{r,d}))t^r
=\sum_{d/r=\mu}\bL^{1/2} \HDT_{r,d}t^r
=(\bL-1)\Log(Q_\mu).
\end{equation}
The next result is an explicit formula for $Q_{r,d}$ mentioned in the introduction, see
\cite{zagier_elementary,laumon_langlands,mozgovoy_moduli}.

\def\fl#1{\lfloor#1\rfloor}
\begin{theorem}\label{zagier}
For any $r,d$, we have
$$
Q_{r,d}=
\sum_{\ov{r_1,\dots,r_k>0}{r_1+\dots+r_k=r}}
\prod_{i=1}^{k-1}
\frac{\bL^{(r_i+r_{i+1})\set{(r_1+\dots+r_i)d/r}}}
{1-\bL^{r_i+r_{i+1}}}Q_{r_1}\dots Q_{r_k},
$$
where $\set x=x-\lfloor x\rfloor$ is the fractional part of $x$ and
$$Q_r=\bL^{(1-g)r^2/2}
\Res_{t=1}\prod_{i=0}^{r-1}Z_C(\bL^i t),
\qquad Z_C(t)
=\sum_{n\ge0}E(S^nC)t^n=\frac{(1-ut)^g(1-vt)^g}{(1-t)(1-uvt)}.$$
\end{theorem}

\begin{remark}
With $r_{\le i}=r_1+\dots+r_i$,
we have
$\sum_{i=1}^{k-1}(r_i+r_{i+1}){r_{\le i}}=(r-r_k)r$,
hence
\begin{equation*}
\sum_{i=1}^{k-1}(r_i+r_{i+1})\set{r_{\le i}d/r}
=(r-r_k)d-\sum_{i=1}^{k-1}(r_i+r_{i+1})\fl{r_{\le i}d/r}\in\bZ.
\end{equation*}
\end{remark}


\appendix
\section{Relative hard Lefschetz theorem}
\label{RHL}
Let $\pi:X\to Y$ be a projective morphism between smooth (connected) algebraic varieties and $\ell\in H^2(X,\bZ(1))$ be the first Chern class
of a relatively ample line bundle on $X$.
It induces the map $\ell:M\to M(1)[2]$ for any $M\in D^b\MHM(X)$.
\begin{theorem}[Relative hard Lefschetz theorem \cite{BBD,saito_modules}]
For a pure $M\in\MHM(X)$, the map
$$\ell^i:H^{-i}\pi_*M\to H^i \pi_*M(i),\qquad i\ge0,$$
is an isomorphism.
\end{theorem}
We define the primitive parts of $\pi_*M$
\begin{equation}
P_i=\Ker(\ell^{i+1}:H^{-i}\pi_*M\to H^{i+2}\pi_*M(i+1)),\qquad i\ge0.
\end{equation}
Applying the isomorphism $\ell^i$, we obtain
\[P_i\iso \Ker(\ell:H^i\pi_*M\to H^{i+2}\pi_*M(1))(i).\]
By the above theorem, there are decompositions
\begin{gather}
M=\bop_{i\ge0}P_i[i]\ts H^*(\bP^i,\bQ),\\
H^kM=\bop_{\ov{0\le j\le i}{-i+2j=k}}P_i(-j)
=\bop_{j\ge0,k}P_{2j-k}(-j)
,\qquad k\in\bZ.
\end{gather}


\begin{lemma}
\label{lm:top prim}
Let $\pi:X\to Y$ be a smooth projective morphism with smooth $Y$ and connected fibres of dimension $n$.
Then the primitive parts of $\pi_*\ICV_X$
satisfy $P_n=H^{-n}\pi_*\ICV_X=\ICV_Y(n/2)$ and $P_i=0$ for $i>n$.
We also have $H^n\pi_*\ICV_X\iso\ICV_Y(-n/2)$.
\end{lemma}
\begin{proof}
The functor $\pi_*$ has amplitude $[-n,n]$, hence $H^{-i}\pi_*\ICV_X=H^i\pi_*\ICV_X=0$ for $i>n$.
Therefore $P_n=H^{-n}\pi_*\ICV_X$.
The map $\bQ_Y\to \pi_*\pi^*\bQ_Y=\pi_*\bQ_X$
induces the map $\ICV_Y\to \pi_*\ICV_X(-n/2)[-n]$,
hence $\ICV_Y\to H^{-n}\pi_*\ICV_X(-n/2)$.
To show that it is an isomorphism, we consider the corresponding map between perverse sheaves
$\IC_Y\to\p H^{-n}\pi_*\IC_X$.
As $\pi$ is smooth, the object $M=\pi_*\IC_X\in D^b_c(\bQ_Y)$ is smooth.
Therefore $\p H^{i}M=(\sH^{i-d_Y}M)[d_Y]$ for all $i\in\bZ$,
where $\sH^iM\in\Sh(\bQ_Y)$ denotes the $i$-th cohomology sheaf.
In particular,
$\p H^{-n}\pi_*\IC_X=(\sH^{-d_X}\pi_*\IC_X)[d_Y]=(\sH^0\pi_*\bQ_X)[d_Y]$.
The map $\IC_Y=\bQ_Y[d_Y]\to (\sH^{0}\pi_*\bQ_X)[d_Y]$
is an isomorphism as the fibres are connected.
\end{proof}


\begin{remark}
If $\pi:X\to Y$ is a $\bP^n$-fibration, then the primitive part $P_n$ of $\pi_*\ICV_X$ is the only nonzero component (we can check this on the fibres).
Therefore
\begin{equation}
\pi_*\ICV_X=\ICV_Y(n/2)[n]\ts H^*(\bP^n,\bQ)
=\ICV_Y\ts H^*(\bP^n,\bQ)\ang{n}.
\end{equation}
\end{remark}

\providecommand{\bysame}{\leavevmode\hbox to3em{\hrulefill}\thinspace}
\providecommand{\href}[2]{#2}

\end{document}